\documentclass[bibtotoc,a4paper,10pt]{article}
\usepackage[top=1.25in,left=1in,right=1in,bottom=1.25in]{geometry}
\usepackage{latexsym}
\usepackage{enumerate,amsmath,amssymb,dsfont,pifont,amsthm}
\usepackage[dvips]{color}
\usepackage{graphicx}
\usepackage[arrow, matrix, curve]{xy}
\usepackage{setspace}
\usepackage{fancyhdr}
\usepackage{amsthm}
\usepackage{latexsym}
\usepackage{dsfont}
\usepackage{bbm}
\usepackage{amssymb}
\usepackage{amsmath}
\usepackage{graphicx}
\usepackage{mathabx}

\newtheorem{theorem}{Theorem}[section]
\newtheorem{prop}[theorem]{Proposition}
\newtheorem{lem}[theorem]{Lemma}
\newtheorem{cor}[theorem]{Corollary}
\newtheorem{rem}[theorem]{Remark}
\newtheorem{exa}[theorem]{Example}
\newtheorem{defi}[theorem]{Definition}

\newtheorem{assum}[theorem]{Assumptions}

\newtheorem{Atheorem}{Theorem}[section]
\newtheorem{Aprop}[Atheorem]{Proposition}
\newtheorem{Alem}[Atheorem]{Lemma}
\newtheorem{Acor}[Atheorem]{Corollary}
\newtheorem{Arem}[Atheorem]{Remark}

\def\gl{\buildrel \rm def\over =}

\DeclareMathOperator{\one}{\mathbbm{1}}

\begin{document}

\allowdisplaybreaks

\title{\bfseries Low Frequency L\'{e}vy Copula Estimation}

\author{%
    \textsc{Christian Palmes}
    \thanks{Lehrstuhl IV, Fakult\"at f\"ur Mathematik, Technische Universit\"at Dortmund,
              D-44227 Dortmund, Germany,
              \texttt{christian.palmes@math.tu-dortmund.de}}
    }


\maketitle

\begin{abstract}
Let $X$ be a $d$-dimensional L\'{e}vy process with L\'{e}vy triplet $(\Sigma,\nu,\alpha)$ and $d\geq 2$. Given the low frequency observations $(X_t)_{t=1,\ldots,n}$, the dependence structure of the jumps of $X$ is estimated. The L\'{e}vy measure $\nu$ describes the average jump behavior in a time unit. Thus, the aim is to estimate the dependence structure of $\nu$ by estimating the L\'{e}vy copula $\mathfrak{C}$ of $\nu$, cf. Kallsen and Tankov \cite{KalTan}.

We use the low frequency techniques presented in a one dimensional setting in Neumann and Rei{\ss} \cite{NeuRei} and Nickl and Rei{\ss} \cite{NicRei} to construct a L\'{e}vy copula estimator $\widehat{\mathfrak{C}}_n$ based on the above $n$ observations. In doing so we prove
$$\widehat{\mathfrak{C}}_n\to \mathfrak{C},\quad n\to\infty$$
uniformly on compact sets bounded away from zero with the convergence rate $\sqrt{\log n}$. This convergence holds under quite general assumptions, which also include L\'{e}vy triplets with $\Sigma\neq 0$ and $\nu$ of arbitrary Blumenthal-Getoor index $0\leq\beta\leq 2$. Note that in a low frequency observation scheme, it is statistically difficult to distinguish between infinitely many small jumps and a Brownian motion part. Hence, the rather slow convergence rate $\sqrt{\log n}$ is not surprising. 

In the complementary case of a compound Poisson process (CPP), an estimator $\widehat{C}_n$ for the copula $C$ of the jump distribution of the CPP is constructed under the same observation scheme. This copula $C$ is the analogue to the L\'{e}vy copula $\mathfrak{C}$ in the finite jump activity case, i.e. the CPP case. Here we establish
$$\widehat{C}_n \to C,\quad n\to\infty$$
with the convergence rate $\sqrt{n}$ uniformly on compact sets bounded away from zero. 

Both convergence rates are optimal in the sense of Neumann and Rei{\ss} \cite{NeuRei}.
\end{abstract}

{\bf Keywords:} statistical inference for jumps, multidimensional L\'{e}vy processes, L\'{e}vy measure, L\'{e}vy copula, copula, low-frequency data \\

{\bf MSC 2010 Classification:} 60G51,  62G20, 62H12\\

\section{Introduction}
In this paper, we analyze the dependence structure of jumps in a multidimensional L\'{e}vy process. This L\'{e}vy process is discretely observed in a \textit{low frequency} scheme. To be more precise, let $(X_t)_{t\geq 0}$ be a $d$-dimensional L\'{e}vy process on a probability space $(\Omega,\mathcal{F},P_{\Sigma,\nu,\alpha})$ with the L\'{e}vy triplet $(\Sigma,\nu,\alpha)$. Based on the equidistant observations $(X_t(\omega))_{t=1,2,\ldots,n}$ for a given realization $\omega\in\Omega$, we intend to estimate the dependence structure of the jumps between the coordinates of $X$. Hence, we have a statistical problem. Next, a rigorous formulation of what is meant by this dependence structure is given:

\subsection{The concept of a L\'{e}vy Copula}
First, we state a well-known concept proposed by Sklar \cite{Skl}. In this context, see also the monograph of Nelsen \cite{Nel}. Given $d$ random variables $Y_1,\ldots, Y_d\,:\,(\widetilde{\Omega},\widetilde{\mathcal{F}},\widetilde{P}) \to (\mathds{R},\mathcal{B})$, a well-known concept to describe the dependence structure within $(Y_1,\ldots,Y_d)$ is provided by its copula $C_{Y_1,\ldots,Y_d}$. This is a $d$ dimensional distribution function with uniform margins, such that we have
$$\widetilde{P}(Y_1\leq y_1, \ldots Y_d\leq y_d) = C_{Y_1,\ldots,Y_d}(\widetilde{P}(Y_1\leq y_1),\ldots, \widetilde{P}(Y_d\leq y_d)),\quad y_1,\ldots y_d\in\mathds{R}.$$
Thus, a copula provides in a certain sense the additional information that is needed to obtain the vector distribution from the marginal distributions. However, in this paper, we deal with a stochastic process 
$$X\,:\,\mathds{R}_+\times \Omega \to \mathds{R}^d,\quad \mathds{R}_+\gl [0,\infty)$$
and not with only finite many real valued random variables. Generally, it is problematic to determine the meaning of the dependence structure of $X$ or even the dependence structure between the coordinates of the jumps of $X$. Nevertheless, in the case of a \textit{L\'{e}vy process} $X$, a natural approach is given by Kallsen and Tankov \cite{KalTan} which uses the fact that $X$ is characterized by its L\'{e}vy triplet $(\Sigma,\nu,\alpha)$. Here, $\nu$ describes the jumps of $X$ in the sense that
$$\nu(A) = E|\{t\in [0,1]\,:\,\Delta X_t\in A\}|,\quad A\in\mathcal{B}(\mathds{R}^d),$$
cf. Sato \cite{Sat}[Theorem 19.2]. Here, we suppose as usual that $X$ has c\`adl\`ag paths and
$$\Delta X_t(\omega) \gl X_t(\omega) - X_{t-}(\omega) \gl X_t(\omega)-\lim_{s\uparrow t,\,s < t} X_s(\omega),\quad \omega\in \Omega.$$
With the jump structure of $X$, we mean the dependence structure of $\nu$. Observe the problem that $\nu$ is in general not a probability measure, so that the copula concept cannot be applied to $\nu$. However, it is at least known that
$$\int_{\mathds{R}^d} |x|^2\wedge 1\,\nu(dx) < \infty$$
holds which implies $\nu(A)<\infty$ for all $A\in\mathcal{B}(\mathds{R}^d)$ with $0\notin \overline{A}$. Thus,  $0\in\mathds{R}^d$ is the only possible singular point. Based on these facts, Kallsen and Tankov \cite{KalTan} introduced the concept of \textit{L\'{e}vy-copulas}. Compare particularly Definition 2.1 - 3.4 and Theorem 3.6 in \cite{KalTan} for the general definition of a L\'{e}vy copula and the statement of Sklars theorem for L\'{e}vy copulas. Our future assumptions in this paper are going to ensure that the L\'{e}vy copula $\mathfrak{C}$ has always the special shape 
\begin{equation} \label{introLevCopEq}
\mathfrak{C}(u,v) =
\begin{cases}
U(U_1^{-1}(u),U_2^{-1}(v)), &u,v > 0, \\
0, & u\leq 0 \text{ or } v\leq 0
\end{cases}
\end{equation}
where 
$$U(x,y)\gl \nu([x,\infty)\times [y,\infty)),\quad U_1(x)\gl \nu([x,\infty)\times\mathds{R}_+),\quad U_2(y)\gl\nu(\mathds{R}_+\times [y,\infty)),\quad x,y\in\mathds{R}_+.$$

At this point we can specify the aim of this paper: Our Aim is to construct and investigate an estimator for the L\'{e}vy copula of $\nu$ based on low frequency observations. The only existing reference in this context is, to our best knowledge, the unpublished paper of Schicks \cite{Sch}. This paper, however, only deals with the compound Poisson process case by means of the deconvolution techniques in Buchmann, Gr\"{u}bel \cite{BucGru}. Note that also B\"{u}cher, Vetter \cite{BucVet} and Laeven \cite{Lae} and Krajina, Laeven \cite{KraLae} have published relevant information close to this subject. Nevertheless, all their approaches work with the following high frequency observation scheme:
$$(X_t)_{t=\Delta_n,2\Delta_n,\ldots, n\Delta_n},\quad \Delta_n\to 0,\quad n\Delta_n\to\infty$$
which results in a completely different analysis than our low frequency observation scheme. Our approach is mostly motivated by Neumann, Rei{\ss} \cite{NeuRei} and Nickl, Rei{\ss} \cite{NicRei} which provide the required low frequency techniques for our needs.

\subsection{Organization of this paper}
In \textit{Section \ref{preSec}}, we state an estimator $\widehat{\nu}_n$ for the L\'{e}vy measure $\nu$, which is based on the low frequency observations $(X_t)_{t=1,2,\ldots,n}$. Our assumptions in this section imply that the second moment of $(\widehat{\nu}_n)$ and $\nu$ exist, i.e.
$$\int_{\mathds{R}^d} |x|^2\,\nu(dx) < \infty,\quad \int_{\mathds{R}^d} |x|^2\,\widehat{\nu}_n(\omega)(dx) <\infty,\quad n\in\mathds{N},\quad \omega\in\Omega.$$
We prove in Theorem \ref{nuWeakConvTh} the weak convergence of Borel measures
$$|x|^2\widehat{\nu}_n(dx) {\buildrel w \over \to } |x|^2\nu(dx),\quad n\to\infty.$$
This means that we have $P_{\Sigma,\nu,\alpha}$-a.s. the convergence
$$\int_{\mathds{R}^d} f(x) |x|^2\,\widehat{\nu}_n(dx) \to \int_{\mathds{R}^d} f(x)|x|^2\,\nu(dx),\quad n\to\infty$$
for all bounded, continuous functions $f$, i.e. $f\in \mathcal{C}_b(\mathds{R}^d)$. Our prove works only under the assumption $\Sigma=0$, i.e. with vanishing Brownian motion part. This is due to the fact that it is statistically hard to distinguish between the small jumps of infinite activity and the Brownian motion part. This issue is stated more precisely in Lemma \ref{uniEstLem}. Note that Neumann and Rei{\ss} \cite{NeuRei} solve this problem in the one dimensional case by estimating
$$\nu_\sigma(dx)\gl \sigma^2\delta_0(dx) + x^2\nu(dx),\quad x\in\mathds{R},\quad \sigma^2\gl \Sigma$$
instead of $\nu$. As a result, the exponent of the characteristic function gets the shape
$$\Psi_{\nu_\sigma,\alpha}(u) = iu\alpha + \int_{\mathds{R}} \frac{e^{iux}-1-iux}{x^2}\,\nu_\sigma(dx),\quad u\in\mathds{R},$$
assuming that the second moment of $\nu$ is finite, compare Section 4 in \cite{NeuRei}. Unfortunately, such a transition from $\nu$ to $\nu_\sigma$ does not work in the multidimensional case $d\geq 2$. Nevertheless, we aim to estimate the L\'{e}vy copula of $\nu$ under assumptions that do \textit{not} exclude the existence of a Brownian motion part, i.e. $\Sigma\neq 0$. Thus, we have to deal somehow with the small jumps of $X$.

This is described in \textit{Section \ref{LevCopSec}}. Here, everything is developed for the case $d=2$. This is only due to a simpler notation of the anyway high technical approach. We first construct an estimator $\widehat{N}_n$ based on the same $n$ equidistant observations as $\widehat{\nu}_n$, such that it holds under certain smoothness and decay conditions on $\nu$
\begin{equation} \label{treatSmallJmpEq}
\sup_{(a,b)\in\mathfrak{R}} \eta(a,b) \left|\nu([a,\infty)\times [b,\infty)) - \widehat{N}_n(a,b)\right| = O_{P_{\Sigma,\nu,\alpha}}\left(\frac{(\log\log n)^2}{\sqrt{\log n}}\right),\quad n\to\infty
\end{equation}
with
$$\mathfrak{R}\gl [0,\infty)^2\backslash\{(0,0)\},\quad \eta(a,b)\gl |(a,b)|^2\wedge |(a,b)|^4.$$
This is proven in Theorem \ref{LevyEstTh}. Note that the right hand side of (\ref{treatSmallJmpEq}) is independent of $(a,b)\in\mathfrak{R}$. If $|(a,b)|\to 0$, $\eta(a,b)\asymp |(a,b)|^4\to 0$ slows down the convergence speed in (\ref{treatSmallJmpEq}). Vice versa $|(a,b)|\to \infty$ implies $\eta(a,b)\asymp |(a,b)|^2\to\infty$ which accelerates the convergence speed. This way of treating the small jumps is sufficient of getting satisfying results concerning the estimation of the L\'{e}vy copula. Note that (\ref{treatSmallJmpEq}) also yields estimations for $\nu([a,\infty)\times\mathds{R}_+)$ resp. $\nu(\mathds{R}_+\times [b,\infty))$ by setting $a>0,\,b=0$ resp. $a=0,\,b>0$. Our assumptions in this section will ensure that the L\'{e}vy copula of $\nu$ can be written in the form (\ref{introLevCopEq}). We are capable to estimate $U,U_1,U_2$ with the use of (\ref{treatSmallJmpEq}). Our intention is to create a plug-in estimator for (\ref{introLevCopEq}), i.e. we also need an estimator for $U_k^{-1}$, $k=1,2$. This basically works by building the pseudo inverse of the estimator of $U_k$. At this point, we again have to pay attention to the small jumps. This inverting procedure is performed by Corollary \ref{LevCopInvCor} which is the stochastic counterpart of Proposition \ref{invProp}. This proposition contains the analysis needed for the inversion operation. Finally, Theorem \ref{LevyGeneralTheorem} states that the resulting plug-in estimator $\widehat{\mathfrak{C}}_n$ converges uniformly on compact sets bounded away from zero with the convergence rate $\sqrt{\log n}$, i.e.: It holds for two arbitrary and fixed numbers $0<a<b<\infty$ the asymptotic
\begin{equation} \label{CopConvEq}
\sup_{a\leq u,v\leq b} |\mathfrak{C}(u,v) - \widehat{\mathfrak{C}}_n(u,v)| = O_{P_{\Sigma,\nu,\alpha}}\left(\frac{(\log\log n)^9}{\sqrt{\log n}}\right).
\end{equation}
The term $(\log\log n)^9$ in (\ref{CopConvEq}) is not relevant in the sense that we have
$$\frac{(\log\log n)^9}{(\log n)^\epsilon} \to 0,\quad n\to\infty$$
for all $\epsilon>0$. Note that the convergence in (\ref{CopConvEq}) holds in a wide, non-pathologic class of L\'{e}vy triplets which contains L\'{e}vy measures of every Blumenthal Getoor index $0\leq \beta\leq 2$, i.e. the test can separate the small jumps from the Brownian motion part in a low frequency setting even in the case $\beta=2$. This is proven in Corollary \ref{BGCor}. Furthermore, observe that $[a,b]\subseteq [0,\infty)^2$ in (\ref{CopConvEq}) is bounded away from zero. However, we have to treat the small jumps tending to zero in order to estimate, for example, $U_k^{-1}$, $k=1,2$, compare the proof of Theorem \ref{LevyGeneralTheorem}. Apart from that, our technical approach would easily yield a similar treatment of $[a,b]^2$ in the case $a\downarrow 0$, $b\uparrow\infty$ as in (\ref{treatSmallJmpEq}). The respective $\eta$ would then, however, depend on $U$, i.e. on $\nu$ which is the unknown estimating entity. Hence, such convergence rates are not statistically feasible and thus, we have not calculated them.

Finally, in \textit{Section \ref{CPPSec}}, we apply the techniques developed in Section \ref{LevCopSec} to the compound Poisson process (CPP) case. For simplicity we assume that the intensity $\Lambda=\nu(\mathds{R}^2)$ is known. We propose an estimator $\widehat{C}_n$ for the copula $C$ of the probability measure $\Lambda^{-1}\nu$, which is based on the same $n$ low frequency observations ($t=1,2,\ldots,n$). For this purpose, we show that everything developed in the previous Section \ref{LevCopSec} also works in this case. Here, we obtain the better and natural convergence rate $\sqrt{n}$ as expected. Namely, we show in Theorem \ref{ComPoiCopTh} that 
$$\sup_{a\leq u,v\leq b} |C(u,v)-\widehat{C}_n(u,v)| = O_{P_{\nu,\alpha}}\left(\frac{(\log n)^{10}}{\sqrt{n}}\right),\quad n\to\infty$$
holds under certain assumptions in the CPP case.

Neumann and Rei{\ss} \cite{NeuRei}[Theorem 4.4] prove in a one dimensional setting that, in the case of a non-vanishing Brownian motion part, a logarithmic convergence rate for estimating $\nu_\sigma$ is optimal. Furthermore, $\sqrt{n}$ is the optimal rate in the CPP case. Hence, the convergence rates of our L\'{e}vy copula estimators can be considered to be optimal in the sense that the optimal rates in the one dimensional setting still hold in the multidimensional setting \textit{and} after an inversion operation.

\subsection{Notations}
In what follows we summarize some frequently used notations.
\begin{figure}[h]
\begin{tabular}{ll}
$\mathds{N} = \{1,2,\ldots\}$ & Natural numbers \\
$\mathds{N}_0 = \mathds{N}\cup \{0\}$ & Natural numbers including zero \\
$\mathds{R}_+ = [0,\infty)$ & Nonnegative real numbers \\
$\mathds{R}^\ast = \mathds{R}\backslash\{0\}$ & Real numbers without zero \\
$\mathds{R}_+^\ast = \mathds{R}^\ast\cap \mathds{R}_+$ & Positive real numbers \\
$\mathfrak{R} = \mathds{R}_+^2\backslash\{(0,0)\}$ & First quadrant in $\mathds{R}^2$ without the origin \\
$z=\Re(z) + i \Im(z)$ & Real and imaginary part of a complex number $z\in\mathds{C}$ \\
$|z|,\,|z|_\infty$ & Euclidean norm and maximum norm of any $z\in\mathds{C}$ \\
$\mathcal{R}_+(z) = \mathcal{R}(z)\vee  0$ & Positive real part of a complex number $z\in\mathds{C}$ \\
$\mathcal{B}(\mathds{R}^d)$ & Borel sets of $\mathds{R}^d$ \\
$\mathcal{C}^k,\,\mathcal{C}^k(\mathds{R}^d)$ & Complex valued functions with $k$ continuous partial derivatives \\
$\mathcal{C}$ & Continuous complex valued functions \\
$\mathcal{C}_b$ & Complex valued continuous, bounded functions \\
$L^1,L^2$ & Function spaces of (quadratic) integrable functions \\
$\mathcal{F}f$ resp. $\mathcal{F}\mu$ & Fourier transform of a suitable function $f$ resp. measure $\mu$\\
$\lambda^d$ & $d$-dimensional Lebesgue measure \\
$g_{a,b}(x_1,x_2) = \frac{\one_{[a,\infty)\times [b,\infty)}(x_1,x_2)}{x_1^4+x_2^4}$ & Useful truncating function for our later needs, \\
& $(x_1,x_2)\in\mathds{R}^2,\quad (a,b)\in\mathfrak{R}.$  
\end{tabular}
\end{figure}

\newpage
Moreover, we will to use the abbreviations
\begin{figure}[h]
\begin{tabular}{ll}
CPP & Compound Poisson process \\
BGi & Blumenthal-Getoor index \\
c\`adl\`ag & right continuous function with existing left limits
\end{tabular}
\end{figure}

\noindent
and write
$$f\lesssim g\quad :\iff\quad \exists C>0\,:\,f(x)\leq Cg(x),\,\forall x\in M$$
for two functions $f,g\,:\,M\to \mathds{R}_+$ on any set $M$. Finally, recall for two families of random variables $(X_n)_{n\in\mathds{N}}$ and $(Y_n)_{n\in\mathds{N}}$ with $P(X_n=0)=0$, $n\in\mathds{N}$ the notation
$$Y_n = O_P(X_n)\quad :\iff\quad \forall\delta>0\,\,\,\exists K>0\,:\, \sup_{n\in\mathds{N}}P\left(\left|\frac{Y_n}{X_n}\right|\geq K\right)\leq \delta.$$

\section{Some preliminary considerations} \label{preSec}
Let $(X_t)_{t\geq 0}$ be a $d$-dimensional L\'{e}vy process on a probability space $(\Omega,\mathcal{F},P_{\Sigma,\nu,\alpha})$ with the L\'{e}vy triplet $(\Sigma,\nu,\alpha)$. Based on the equidistant observations $(X_t(\omega))_{t=1,2,\ldots,n}$ for some fixed path $\omega\in\Omega$, we intend to estimate the L\'{e}vy triplet $(\Sigma,\nu,\alpha)$. First of all note that it is statistical not possible to distinguish  between the existence of a Brownian motion part and an accumulation of infinitely many jumps in a uniform consistent way. This is explained by the following lemma which is a generalization of Remark 3.2 in Neumann, Rei{\ss} \cite{NeuRei}. We give a detailed proof in order to get a starting point in this topic.
\begin{lem} \label{uniEstLem}
Set $d=1$ and write $\sigma=\Sigma$. Then we have
$$\sup_{\sigma,\nu,\alpha} P_{\sigma,\nu,\alpha} \left(|\widehat{\sigma} - \sigma| \geq \frac{1}{2}\right) \geq \frac{1}{2}$$
where $\widehat{\sigma}$ is any real valued random variable. For example $\widehat{\sigma}$ can be any estimator based on the above low frequency observations.
\end{lem}
\begin{proof}
Denote with $P_m,\,m\geq 1$ the $2-\frac{1}{m}$ symmetric stable law, i.e. the law with the characteristic function $$\varphi_m(u) = e^{-\frac{|u|^{2-\frac{1}{m}}}{2}}.$$ 
As $\varphi_m$ is Lebesgue integrable, $P_m$ has the Lebesgue density 
$$f_m(x)=\frac{1}{2\pi} \int e^{-iux} \varphi_m(u)\,du.$$ 
Now consider the total variation (TV) between $P_m$ and $P_\infty\gl \text{N}(0,1)$. Scheff\'{e}'s Lemma yields
\begin{equation} \label{TVequ}
\|P_m - P_\infty\|_{TV} = \frac{1}{2}\int|f_m(x) - f_\infty(x)|\,dx.
\end{equation}
We have
$f_m \to f_\infty$ pointwise because of
$$|f_m(x) - f_\infty(x)| \leq \int|\varphi_m(u)-\varphi_\infty(u)|\,du$$
and the integrable $L^1(\lambda^1)$ majorant
$$|\varphi_m(u) - \varphi_\infty(u)|\leq 2e^{-\frac{u^2\wedge |u|}{2}},\quad m\in\mathds{N}.$$
This implies together with
$$\int f_m(x)\,dx = \int f_\infty(x)\,dx = 1$$
and a theorem of Riesz, cf. \cite{Bau}[Theorem 15.4] that $f_m\to f_\infty$ in $L^1(\lambda)$, i.e. with (\ref{TVequ})
$$\|P_m-P_\infty\|_{TV} \to 0,\quad m\to\infty.$$
Now consider the two sets
$$A_0\gl \left\{|\widehat{\sigma}| \geq \frac{1}{2} \right\},\quad A_1\gl \left\{|\widehat{\sigma}-1| \geq \frac{1}{2} \right\}.$$
Fix $\epsilon>0$ and choose $m$ large enough, such that $\|P_m-P_\infty\|_{TV}<\epsilon$. Note that the Brownian part $\sigma$ of $P_m$ is zero and the Brownian part of $P_\infty$ is one. Assume that
$$P_m\left(|\widehat{\sigma} - \sigma| \geq \frac{1}{2}\right) = P_m(A_0) < \frac{1}{2}.$$
Then we have $P_\infty(A_0)<\frac{1}{2}+\epsilon$ which yields because of $A_0\cup A_1 = \Omega$
$$P_\infty(A_1)+\frac{1}{2}+\epsilon > P_\infty(A_0)+P_\infty(A_1) \geq P_\infty(A_0\cup A_1) = 1.$$
This results in
$$P_\infty\left(|\widehat{\sigma}-\sigma|\geq\frac{1}{2}\right) = P_\infty(A_1)\geq\frac{1}{2}-\epsilon.$$
The lemma is proven since $\epsilon>0$ was chosen arbitrarily. 
\end{proof}

We denote by
$$\widehat{\varphi}_n(u) \gl \frac{1}{n}\sum_{t=1}^n e^{i\left<u,X_t-X_{t-1}\right>},\quad u\in\mathds{R}^d$$
the empirical characteristic function of the increments and write, furthermore,
$$\varphi_{\Sigma,\nu,\alpha}(u)\gl E_{\Sigma,\nu,\alpha}\left(e^{i\left<u,X_1\right>}\right),\quad u\in\mathds{R}^d.$$
Next let $w\,:\,\mathds{R}^d\to\mathds{R}^{>0}$ denote a weight function which is specified later. For the following we only require that $w$ is bounded and vanishes at infinity. Define the weighted supremum
$$\|\psi\|_{L^\infty(w)} \gl \sup_{u\in\mathds{R}^d}\{w(u)|\psi(u)|\}$$
for mappings $\psi\,:\,\mathds{R}^d\to\mathds{C}$. The following proposition is needed in order to prove Theorem \ref{nuWeakConvTh}. This theorem yields a consistent estimator $\widehat{\nu}_n$ for the L\'{e}vy measure $\nu$ as described in the introductory section.

\begin{prop} \label{weightpointprop}
For every $(\Sigma,\nu,\alpha)$ with $E_{\Sigma,\nu,\alpha}|X_1|^4 < \infty$, there exists a $P_{\Sigma,\nu,\alpha}$ negligible set $N$, such that
\begin{equation} \label{pointconveq}
\left\|\frac{\partial^l}{\partial u_k^l}\left(\widehat{\varphi}_n(u)-\varphi_{\Sigma,\nu,\alpha}(u)\right)\right\|_{L^\infty(w)}\to 0,\quad n\to\infty,\quad l=0,1,2,\quad 1\leq k\leq d
\end{equation}
holds on $N^c$.
\end{prop}
\begin{proof}
Suppose for example $k=1,\,l=2$. Then (\ref{pointconveq}) claims that
$$\sup_{u\in\mathds{R}^d}\left\{w(u)\left|\frac{1}{n}\sum_{t=1}^n(X_{t,1}(\omega)-X_{t-1,1}(\omega))^2e^{i\left<u,X_t(\omega)-X_{t-1}(\omega)\right>}-E_{\Sigma,\nu,b}X_{1,1}^2e^{i\left<u,X_1\right>}\right|\right\} \to 0,\quad n\to\infty$$
holds for all $\omega\in N^c$, where $X_{t,1}\in\mathds{R}$ denotes the first component of $X_t\in\mathds{R}^d$.

The proof can be done in two steps. The first one deals with the continuity of $\widehat{\varphi}_n$ and $\varphi_{\Sigma,\nu,\alpha}$. Doing so, we can replace $\mathds{R}^d$ by a countable, unbounded set. Finally a second step makes use of the weight function $w$ in order to treat the unboundedness of the remaining countable set in the first step. We omit a detailed proof since it can be done with straightforward techniques.  
\end{proof}

\begin{assum} \label{ConsEstAssum}
The L\'{e}vy triplet $(\Sigma,\nu,\alpha)$ has a vanishing Brownian part and possesses a finite fourth moment, i.e.
$$\int|x|^4 \nu(dx) < \infty,\quad \Sigma = 0.$$
\end{assum}

The above Assumptions \ref{ConsEstAssum} are motivated by Lemma \ref{uniEstLem} and Proposition \ref{weightpointprop}. Note that given the Assumptions \ref{ConsEstAssum} we do not need a truncation function in the representation of the characteristic function of the L\'{e}vy process, i.e. we are going to use the representation
\begin{equation} \label{charrepeq}
\varphi_{\nu,\alpha}(u) = \exp{\left(i\left<u,\alpha\right> + \int_{\mathds{R}^d}(e^{i\left<u,x\right>} - 1 - i\left<u,x\right>)\,\nu(dx)\right)}.
\end{equation}
Define the following metric on $\mathcal{C}^2(\mathds{R}^d):$
$$d^{(2)}(\varphi_1,\varphi_2)\gl \|\varphi_1-\varphi_2\|_{L^\infty(w)} + \sum_{k=1}^d\left\|\frac{\partial}{\partial u_k}\left(\varphi_1-\varphi_2\right)\right\|_{L^\infty(w)} + \sum_{k=1}^d\left\|\frac{\partial^2}{\partial u_k^2}\left(\varphi_1-\varphi_2\right)\right\|_{L^\infty(w)}.$$
Let Assumptions \ref{ConsEstAssum} hold. Given the equidistant observations as described at the beginning of this section, we introduce the minimum distance estimator $(\widehat{\nu}_n,\widehat{\alpha}_n)$ via
\begin{equation} \label{mindisteq}
d^{(2)}(\widehat{\varphi}_n,\varphi_{\widehat{\nu}_n,\widehat{\alpha}_n}) \leq \inf_{\nu,\alpha} d^{(2)}(\widehat{\varphi}_n,\varphi_{\nu,\alpha}) + \delta_n,\quad n\in\mathds{N},\quad\omega\in\Omega
\end{equation}
for a given sequence $\delta_n\downarrow 0$. This means that $(\widehat{\nu}_n,\widehat{\alpha}_n)$ are chosen in such a way that (\ref{mindisteq}) and the Assumptions \ref{ConsEstAssum} are fulfilled. This is exactly the multidimensional variant of (2.3) in Neumann, Rei{\ss} \cite{NeuRei}.

However, note that the estimator $\widehat{\nu}_n$ is not directly feasible because the infimum in (\ref{mindisteq}) is taken over the infinite set of all $\nu,\alpha$. This problem is solved in the next section by constructing and analyzing a more subtle and feasible estimator.

\begin{theorem} \label{nuWeakConvTh}
Let Assumptions \ref{ConsEstAssum} hold and let $w$ be continuous and vanishing at infinity. Then we have $P_{\nu,\alpha}$-a.s. the weak convergence
$$|x|^2\widehat{\nu}_n(dx)\, {\buildrel w \over \to}\, |x|^2\nu(dx).$$ 
\end{theorem}
\begin{proof}
Proposition \ref{weightpointprop} yields
\begin{eqnarray}
d^{(2)}(\varphi_{\widehat{\nu}_n,\widehat{\alpha}_n},\varphi_{\nu,\alpha}) &\leq& d^{(2)}(\varphi_{\widehat{\nu}_n,\widehat{\alpha}_n},\widehat{\varphi}_n) + d^{(2)}(\widehat{\varphi}_n,\varphi_{\nu,\alpha}) \nonumber \\ &\leq&
2d^{(2)}(\widehat{\varphi}_n,\varphi_{\nu,\alpha}) + \delta_n \nonumber\\ &\to& 0,\quad P_{\nu,\alpha}\text{-a.s.} \label{triangleeq}
\end{eqnarray}
It follows for any compact $K\subseteq \mathds{R}^d$
\begin{equation} \label{wconveq1}
\int_K |\varphi_{\widehat{\nu}_n,\widehat{\alpha}_n}(u) - \varphi_{\nu,\alpha}(u)|\,\lambda^d(du) \leq \left(\inf_{u\in K} w(u) \right)^{-1} \lambda^d(K) \|\varphi_{\widehat{\nu}_n,\widehat{\alpha}_n} - \varphi_{\nu,\alpha} \|_{L^\infty(w)},\quad \omega\in\Omega.
\end{equation}
Note that
$$\|\varphi_{\widehat{\nu}_n,\widehat{\alpha}_n} - \varphi_{\nu,\alpha}\|_{L^\infty(w)} \leq d^{(2)}(\varphi_{\widehat{\nu}_n,\widehat{\alpha}_n},\varphi_{\nu,\alpha}) \to 0,\quad P_{\nu,\alpha}\text{-a.s.}$$
Fix any $v=(v_1,\ldots,v_d)\in\mathds{R}^d$ and define
$$K_n=K_n(\omega)\gl [v_1,v_1+\epsilon_n] \times \ldots \times [v_d,v_d+\epsilon_n] \subseteq \mathds{R}^d,\quad n\in\mathds{N}$$
for a sequence $(\epsilon_n)_n$ with $\epsilon_n\downarrow 0$, which may depend on $\omega\in\Omega$. Then we have because of the strict positivity and continuity of $w$
$$C\gl \sup_{n\in\mathds{N}}\left(\inf_{u\in K_n} w(u)\right)^{-1} = \left(\inf_{u\in K_1} w(u)\right)^{-1} < \infty.$$
From (\ref{wconveq1}), it follows that
\begin{equation} \label{wconveq2}
\frac{1}{\epsilon_n^d} \int_{K_n} |\varphi_{\widehat{\nu}_n,\widehat{\alpha}_n}(u)-\varphi_{\nu,\alpha}(u)|\,\lambda^d(du) \leq C\|\varphi_{\widehat{\nu}_n,\widehat{\alpha}_n}-\varphi_{\nu,\alpha}\|_{L^\infty(w)} \to 0,\quad P_{\nu,\alpha}\text{-a.s.}
\end{equation}
As $u\mapsto |\varphi_{\widehat{\nu}_n,\widehat{\alpha}_n}(u)-\varphi_{\nu,\alpha}(u)|$ is continuous, we can choose $\epsilon_n = \epsilon_n(\omega)$ small enough that
$$\left|\frac{1}{\epsilon_n^d} \int_{K_n} |\varphi_{\widehat{\nu}_n,\widehat{\alpha}_n}(u)-\varphi_{\nu,\alpha}(u)|\,\lambda^d(du)
- |\varphi_{\widehat{\nu}_n,\widehat{\alpha}_n}(v)-\varphi_{\nu,\alpha}(v)| \right| \leq \frac{1}{n},\quad \omega\in\Omega.$$
This yields together with (\ref{wconveq2})
$$\varphi_{\widehat{\nu}_n,\widehat{\alpha}_n}(v) \to \varphi_{\nu,\alpha}(v)\quad P_{\nu,\alpha}\text{-a.s.}$$
and the negligible set is independent of $v\in\mathds{R}^d$. As a result, L\'{e}vy-Cram\'{e}r yields that we have $P_{\nu,\alpha}$-a.s. 
the weak convergence
$$P_{\widehat{\nu}_n,\widehat{\alpha}_n}^{X_1}\,\, {\buildrel w \over \to}\,\, P_{\nu,\alpha}^{X_1}.$$
With (\ref{charrepeq}), we have
\begin{eqnarray*}
\int_{\mathds{R}^d} |x|^2\,\widehat{\nu}_n(dx) &=& \sum_{k=1}^d \left(\frac{\partial \varphi_{\widehat{\nu}_n,\widehat{\alpha}_n}}{\partial u_k} (0)\right)^2 - \frac{\partial^2 \varphi_{\widehat{\nu}_n,\widehat{\alpha}_n}}{\partial u_k^2} (0) \\&\to&
\sum_{k=1}^d \left(\frac{\partial \varphi_{\nu,\alpha}}{\partial u_k} (0)\right)^2 - \frac{\partial^2 \varphi_{\nu,\alpha}}{\partial u_k^2} (0) = \int_{\mathds{R}^d} |x|^2\, \nu(dx),\quad P_{\nu,\alpha}\text{-a.s.}
\end{eqnarray*}
because of (\ref{triangleeq}).
It therefore suffices to prove the vague convergence, cf. Chung \cite{Chu}.
$$|x|^2 \widehat{\nu}_n(dx)\, {\buildrel v \over \to}\, |x|^2\nu(dx),\quad P_{\nu,\alpha}\text{-a.s.}$$
Let $f\,:\,\mathds{R}^d\to \mathds{R}$ be a continuous function with compact support. It remains to verify
\begin{equation} \label{wconveq3}
\int_{\mathds{R}^d} f(x)|x|^2\widehat{\nu}_n(dx) \to \int_{\mathds{R}^d} f(x)|x|^2\, \nu(dx),\quad P_{\nu,\alpha}\text{-a.s.}
\end{equation}
For this purpose set
\begin{equation*}
h_\epsilon(r) \gl
\begin{cases}
0, &0\leq r < \frac{\epsilon}{2} \\
\frac{2}{\epsilon}\left(r-\frac{\epsilon}{2}\right), &\frac{\epsilon}{2}\leq r < \epsilon \\
1, &r\geq\epsilon
\end{cases},\quad r\geq 0
\end{equation*}
and 
$$g_\epsilon(x)\gl h_\epsilon(|x|),\quad x\in\mathds{R}^d.$$
Fix any $\delta>0$ and choose $\epsilon=\epsilon(\omega)>0$ small enough that
$$\|f\|_\infty \sup_{n\in\mathds{N}} \int_{\{|x|\leq\epsilon\}} |x|^2\, \widehat{\nu}_n(dx) < \delta,\quad \|f\|_\infty \int_{\{|x|\leq \epsilon\}} |x|^2\, \nu(dx) < \delta, \quad P_{\nu,\alpha}\text{-a.s.}$$
This is possible because of Sato \cite{Sat} Theorem 8.7.(2) and our assumption $\Sigma = \widehat{\Sigma}_n = 0$ on $\Omega$. Based on this, we obtain
\begin{equation} \label{wconveq4}
\int_{\mathds{R}^d}f(x)(1-g_{\epsilon}(x))|x|^2\widehat{\nu}_n(dx) < \delta,\quad \int_{\mathds{R}^d}f(x)(1-g_\epsilon(x))|x|^2\,\nu(dx) < \delta,\quad P_{\nu,\alpha}\text{-a.s.}
\end{equation}
On the other hand, $x\mapsto f(x)|x|^2g_\epsilon(x),\,x\in\mathds{R}^d$ is a bounded, continuous function which vanishes on a neighborhood of zero. Thus, Sato \cite{Sat}[Theorem 8.7.(1)] yields
$$\int_{\mathds{R}^d} f(x) g_\epsilon(x)|x|^2 \widehat{\nu}_n(dx) \to \int_{\mathds{R}^d} f(x)g_\epsilon(x) |x|^2 \,\nu(dx),\quad P_{\nu,\alpha}\text{-a.s.}$$
Hence, together with (\ref{wconveq4}), we obtain
$$\limsup_n\left|\int_{\mathds{R}^d} f(x)|x|^2\widehat{\nu}_n(dx) - \int_{\mathds{R}^d} f(x)|x|^2\, \nu(dx) \right| \leq 2\delta,\quad P_{\nu,\alpha}\text{-a.s.}$$
This proves (\ref{wconveq3}) because $\delta>0$ was chosen arbitrarily.
\end{proof}

Set $Z_t = X_t - X_{t-1},\,t\in\mathds{N}$ and define
$$A_n(u) \gl n^{-\frac{1}{2}} \sum_{t=1}^n \left(e^{i\left<u,Z_t\right>} - E\left(e^{i\left<u,Z_1\right>}\right)\right),\quad u\in\mathds{R}^d,\quad n\in\mathds{N}.$$
Next, we aim to establish a similar result as in Proposition \ref{weightpointprop} which will be useful for the next Section \ref{LevCopSec}. The difference to the previous result is that we are not interested in a $P_{\Sigma,\nu,\alpha}$-a.s. result, but in finding an upper bound for the expectation values as stated in the next theorem. Furthermore, $A_n(u)$ is scaled with $n^{-\frac{1}{2}}$ and not with $n^{-1}$ as in Proposition \ref{weightpointprop}. Therefore it is not surprising that we have to make some further restrictions to the weight function $w$. To be more precise, we choose $w$ as
$$w(u) = (\log(e+|u|))^{-\frac{1}{2}-\delta},\quad u\in\mathds{R}^d$$
for some fixed $\delta>0$. This is the natural generalization to $d$ dimensions of the weight function in \cite{NeuRei}. The proof of the following theorem is postponed to the appendix. 

\begin{theorem} \label{Eboundth}
Let $(\Sigma,\nu,\alpha)$ be a L\'{e}vy triplet, such that $E|X_1|^{8+\gamma}<\infty$ holds for some $\gamma>0$. Then we have
$$\sup_{n\geq 1} E_{\Sigma,\nu,\alpha} \left\|\frac{\partial^l}{\partial u_k^l} A_n(u)\right\|_{L^\infty(w)} < \infty$$
for all $0\leq l \leq 4$ and $1\leq k\leq d$.
\end{theorem}

Further, define a $\mathcal{C}^4(\mathds{R}^d)$ metric
$$d(\varphi_1,\varphi_2)\gl d^{(4)}(\varphi_1,\varphi_2)\gl + \sum_{j=1}^4\sum_{k=1}^d\left\|\frac{\partial^j}{\partial u_k^j} (\varphi_1-\varphi_2)\right\|_{L^\infty(w)}.$$
Then, as a direct consequence of Theorem \ref{Eboundth}, it holds the following statement:

\begin{cor} \label{Eboundcor}
Let $(\Sigma,\nu,\alpha)$ be a L\'{e}vy triplet with $E|X_1|^{8+\gamma}<\infty$ for some $\gamma>0$. Then we have
$$E_{\Sigma,\nu,\alpha} d(\widehat{\varphi}_n,\varphi_{\Sigma,\nu,\alpha}) = O\left(n^{-\frac{1}{2}}\right),\quad n\to\infty.$$
\end{cor}

\section{Nonparametric low frequency L\'{e}vy copula estimation} \label{LevCopSec}
We denote with $\mathcal{F}$ the Fourier transform of a function or a finite measure. To be more precise, we set for $u\in\mathds{R}^d$
$$(\mathcal{F}f)(u)\gl \int_{\mathds{R}^d} e^{i\left<u,x\right>}f(x)\,\lambda^d(dx),\quad f\in L^1(\lambda^d)$$
and
$$(\mathcal{F}\mu)(u)\gl \int_{\mathds{R}^d} e^{i\left<u,x\right>}\,\mu(dx)$$
where $\mu$ denotes a finite positive measure on the space $(\mathds{R}^d,\mathcal{B}(\mathds{R}^d))$.

As described in the introductory section, we aim to estimate the L\'{e}vy measure $\nu$ in order to construct a L\'{e}vy copula estimator. Motivated by Nickl and Rei{\ss} \cite{NicRei}, we do not estimate directly $\nu$, but a smoothed version of $\nu$. The statistical estimation of this smoothed version is investigated in the proof of Theorem \ref{LevyEstTh}. An upper bound of the error which we make by using a smoothed version of $\nu$ instead of $\nu$ itself, is calculated in Lemma \ref{convErrorLem}.

\subsection{The Assumptions}
We consider the convolution of $\nu$ with a Kernel $K$ in order to get such a smoothed version of $\nu$, cf. Lemma \ref{convErrorLem}. Such a Kernel, of course, has  to fulfill some assumptions which are stated next: 

\begin{assum} \label{kernelAssum} \rm
Let $K\,:\,\mathds{R}^2\to\mathds{R}_+$ be a kernel function with the properties
\begin{itemize}
\item[(i)]
$K\in L^1(\mathds{R}^2)\cap L^2(\mathds{R}^2),\quad \int_{\mathds{R^2}} K(x)\,\lambda^2(dx) = 1$
\item[(ii)]
$\text{supp}(\mathcal{F}K)\subseteq [-1,1]^2$
\item[(iii)]
$u\,\mapsto\,(\mathcal{F}K)(u)\text{ is Lipschitz continuous.}$
\end{itemize}
\end{assum}

It is natural to consider L\'{e}vy processes in the Fourier space because of the L\'{e}vy-Khintchine formula. From this point of view, Assumption \ref{kernelAssum} (ii) is particularly useful because it provides compact support for many important integrands we use.
\begin{exa} \rm
In the following example, we state a kernel function $K$ which fulfills the Assumptions \ref{kernelAssum}. Set for this purpose
\begin{equation*}
\begin{array}{llcl}
K_1 \,:\,&\mathds{R} &\to& \mathds{R}_+ \\
&x_1 &\mapsto& 
\begin{cases}
\frac{2}{\pi}\left(\frac{\sin\left(\frac{x_1}{2}\right)}{x_1}\right)^2, & x_1 \neq 0 \\
\frac{1}{2\pi}, & x_1=0
\end{cases}
\end{array}
\end{equation*}
and
$$K(x)\gl K_1(x_1)\cdot K_1(x_2),\quad x=(x_1,x_2)\in\mathds{R}^2.$$
Since $K_1$ is continuous and $K_1(x_1) = O(x_1^{-2})$, $|x_1|\to\infty$,
it follows $K_1\in L^1(\mathds{R})\cap L^2(\mathds{R})$ and therefore $K\in L^1(\mathds{R}^2)\cap L^2(\mathds{R}^2)$.
A straightforward calculation yields
$$(\mathcal{F} K_1)(u_1) = (1-|u_1|)\one_{(-1,1)}(u_1),\quad u_1\in\mathds{R}.$$
This implies
$$(\mathcal{F}K)(u) = (\mathcal{F} K_1)(u_1)\cdot(\mathcal{F}K_1)(u_2),\quad u=(u_1,u_2)\in\mathds{R}^2.$$
Note that $K$ fulfills the desired conditions since
$$\int_{\mathds{R}^2} K(x) \lambda^2(x) = (\mathcal{F}K)(0) = (\mathcal{F}K_1)(0)\cdot(\mathcal{F}K_1)(0) = 1$$
and that we have for $u,v\in\mathds{R}^2$
\begin{eqnarray*}
|\mathcal{F}K(u)-\mathcal{F}K(v)| &\leq& |\mathcal{F}K_1(u_1)(\mathcal{F}K_1(u_2)-\mathcal{F}K_1(v_2))| + |\mathcal{F}K_1(v_2)(\mathcal{F} K_1(u_1) - \mathcal{F} K_1(v_1))| \\&\leq&
|\mathcal{F}K_1(u_2) - \mathcal{F}K_1(v_2)| + |\mathcal{F}K_1(u_1) - \mathcal{F}K_1(v_1)| \\ &\leq&
|u_2-v_2| + |u_1-v_1| \\&\leq&
\sqrt{2} |u-v|.
\end{eqnarray*}
Thus, the Assumptions \ref{kernelAssum} are fulfilled.
\end{exa}

Next, set for $h>0$
$$K_h(x)\gl h^{-2} K(h^{-1}x) = (h^{-1}K_1(h^{-1}x_1))\cdot(h^{-1}K_1(h^{-1}x_2)),\quad x\in\mathds{R}^2$$
and observe that standard results from Fourier analysis yield
$$(\mathcal{F}K_h)(u) = (\mathcal{F}K)(hu),\quad u\in\mathds{R}^2.$$
Recall the notations
$$U(x,y) = \nu([x,\infty)\times [y,\infty)),\quad U_1(x) = \nu([x,\infty)\times\mathds{R}_+),\quad U_2(y) = \nu(\mathds{R}_+\times [y,\infty)),\quad x,y\in\mathds{R}_+$$
and 
$$\mathfrak{R} = \mathds{R}_+^2\backslash \{(0,0)\},\quad g_{a,b}(x)\gl \frac{1}{x_1^4+x_2^4}\one_{[a,\infty)\times [b,\infty)}(x_1,x_2),\quad (a,b)\in\mathfrak{R},\quad x\in\mathds{R}^2.$$

\begin{assum} \label{LevyAssum}
Next we state some assumptions concerning the L\'{e}vy measure $\nu$:
\begin{itemize}
\item[(i)]
$\nu\left(\mathds{R}^2\backslash[0,\infty)^2\right) = 0, \quad \text{i.e. only positive jumps},$
\item[(ii)]
$\exists \gamma> 0\,:\,\int |x|^{8+\gamma}\,\nu(dx) < \infty, \quad \text{i.e. finite 8+$\gamma$-th moment},$
\item[(iii)]
$\mathcal{F}((x_1^4+x_2^4)\nu)(u) \lesssim (1+|u_1|)^{-1}(1+|u_2|)^{-1},\quad u\in\mathds{R}^2,$
\item[(iv)]
$U_k\,:\, (0,\infty) \to (0,\infty)$ is a $\mathcal{C}^1$-bijection with $U_k'<0$ and
\begin{equation} \label{UDerCond}
\inf_{0<x_k\leq 1}|U_k'(x_k)| > 0,\quad \sup_{x_k>0} (1\wedge x_k^3)|U_k'(x_k)|<\infty,\quad k=1,2.
\end{equation}
\end{itemize}
\end{assum}

\begin{rem} \rm
Assumption \ref{LevyAssum} (i) assures that there are no negative jumps. This simplifies the shape of the L\'{e}vy copula of $\nu$, cf. (\ref{introLevCopEq}) and serves to keep the technical overhead as small as possible. (ii) is required since we aim to use the statement of Corollary \ref{Eboundcor}. (iii) is perhaps the most non-transparent assumption. It guarantees a certain decay behavior of some integrands in the Fourier space. Finally, (iv) is needed for the construction of the pseudo inverse in order to estimate the L\'{e}vy copula of $\nu$, which is our final goal.

The next proposition and corollary state that these assumptions are not very restrictive in our context, compare the discussion in the introductory section 1.2.
\end{rem}

\begin{prop} \label{BGProp}
Let $f\,:\,\mathds{R}_+^2\to\mathds{R}_+$ be a continuous function with the properties
\begin{itemize}
\item[(i)]
$f(x)>0,\quad x\in(\{0\}\times\mathds{R}_+^\ast)\times(\mathds{R}_+^\ast\times\{0\})$,
\item[(ii)]
$|x|^2\lesssim f(x) \lesssim (\log|x|)^{-2},\quad x\in\mathds{R}_+^2\,:\,|x|\leq\frac{1}{2}$,
\item[(iii)]
$f(x)\lesssim (1+|x|)^{-(6+\epsilon)},\quad x\in \mathds{R}_+^2$
\end{itemize}
for some $\epsilon>0$. Then, 
$$\nu(dx)\gl \one_{\mathfrak{R}}(x)(x_1^4+x_2^4)^{-1}f(x) \lambda^2(dx)$$
is a L\'{e}vy measure and fulfills the Assumptions \ref{LevyAssum} (i), (ii) and (iv).
\end{prop}

\begin{proof}
First, observe that $\nu$ is a L\'{e}vy measure since
$$\int_{\mathds{R}_+^2} |x|^2\,\nu(dx) \lesssim \int_0^{\frac{1}{2}} r^2r^{-4} (\log r)^{-2} r\,dr + \int_{\frac{1}{2}}^\infty r^2 r^{-4} (1+r)^{-(6+\epsilon)}r\,dr <\infty$$
holds. Next, we turn to the claimed Assumptions \ref{LevyAssum} (i), (ii) and (iv).
\begin{itemize}
\item[(i)]
This is obviously true due to $\mathfrak{R} \subseteq \mathds{R}_+^2$.
\item[(ii)]
Note that we have
$$\int_{\mathds{R}_+^2} |x|^{8+\frac{\epsilon}{2}}\,\nu(dx) \lesssim \int_{\mathds{R}_+^2} |x|^{4+\frac{\epsilon}{2}}f(x)\,\lambda^2(dx) \lesssim \int_0^\infty r^{4+\frac{\epsilon}{2}}(1+r)^{-(6+\epsilon)} r\,dr <\infty.$$
Hence, Sato \cite{Sat}[Theorem 25.3] yields that the $(8+\gamma)$-th moment with $\gamma\gl\frac{\epsilon}{2}>0$ of the corresponding L\'{e}vy process exists.

\item[(iv)]
First, observe
\begin{equation} \label{USurEq}
0=\lim_{x_1\uparrow\infty} U_1(x_1) \leq \lim_{x_1\downarrow 0} U_1(x_1) = \infty
\end{equation}
because of 
$$[x_1,\infty)\times\mathds{R}_+ \downarrow \emptyset,\quad x_1\uparrow\infty$$
and
$$\nu(\mathds{R}_+^2) = \int_{\mathds{R}_+^2}(x_1^4+x_2^4)^{-1}f(x)\,\lambda^2(dx) \gtrsim \int_0^{\frac{1}{2}} r^{-4}\cdot r^2 r\,dr = \infty.$$
Next, $U_1(x_1)>0$ follows from $f(\cdot,0)>0$ on $(0,\infty)$ and the continuity of $f$. Hence (\ref{USurEq}) yields that $U_1\,:\,(0,\infty)\to (0,\infty)$ is a surjection. Furthermore, we have for $x_1>0$
\begin{equation} \label{UDiffEq}
U_1'(x_1) = \frac{\partial}{\partial x_1} \int_{x_1}^\infty \int_0^\infty (y_1^4+y_2^4)^{-1} f(y)\,dy_2\,dy_1 = - \int_0^\infty (x_1^4+y_2^4)^{-1} f(x_1,y_2)\,dy_2.
\end{equation}
Again, due to the continuity of $f$ and $f(\cdot,0)>0$, this implies $U_1'<0$ on $(0,\infty)$. Hence, $U_1$ is also injective, i.e. a bijection. Finally, (\ref{UDiffEq}) also implies (\ref{UDerCond}) for $k=1$. Observe for this purpose
$$|U_1'(x_1)|\leq \|f\|_{\infty} \int_0^\infty (x_1^4 + y_2^4)^{-1}\,dy_2 = \frac{\|f\|_\infty \sqrt{2} \pi}{4 x_1^3} \lesssim x_1^{-3},\quad x_1>0$$
and, with the use of Fatou's Lemma,
$$\liminf_{x_1\to 0} |U_1'(x_1)| \geq \int_0^\infty \liminf_{x_1\to 0} [(x_1^4+y_2^4)^{-1}f(x_1,y_2)]\,dy_2 = \int_0^\infty y_2^{-4} f(0,y_2)\,dy_2 > 0.$$
Now, (iv) is verified since nothing changes with $U_2$ instead of $U_1$.
\end{itemize}
\end{proof}

\begin{cor} \label{BGCor}
It exists for every $0\leq\beta\leq 2$ a L\'{e}vy measure $\nu_\beta$ with Blumenthal Getoor index (BGi) $\beta$, such that $\nu_\beta$ fulfills the Assumptions \ref{LevyAssum}.
\end{cor}
\begin{proof}
We treat the cases $0\leq \beta<2$ and $\beta=2$ separately in two steps:

\vspace{0,2cm}
{\sc step 1.}
The case $0\leq\beta<2$. Set
$$f_\beta(x)\gl r^{2-\beta} e^{-r},\quad r\gl |x|,\quad x\in\mathds{R}_+^2$$
and
$$\nu_\beta(dx)\gl \one_{\mathfrak{R}}(x)(x_1^4+x_2^4)^{-1}f_\beta(x)\,\lambda^2(dx).$$
An easy calculation yields that $\nu_\beta$ is a L\'{e}vy measure of BGi $\beta$. Furthermore, the Assumptions \ref{LevyAssum} (i), (ii) and (iv) are fulfilled because of Proposition \ref{BGProp} and
$$|x|^2\lesssim |x|^{2-\beta}e^{-|x|} \lesssim (\log |x|)^{-2},\quad x\in\mathds{R}_+^2\,:\,|x|\leq\frac{1}{2},\quad 0\leq\beta<2.$$
Next, we show that $\nu_\beta$ fulfills Assumption \ref{LevyAssum} (iii). Note for this after a straightforward calculation the equations
\begin{eqnarray*}
\frac{\partial f_\beta}{\partial x_k}(x) &=& x_k\left((2-\beta)r^{-\beta} - r^{1-\beta}\right)e^{-r},\quad k=1,2, \\
\frac{\partial^2 f_\beta}{\partial x_1\partial x_2}(x) &=& x_1x_2\left(\beta(\beta-2)r^{-\beta-2} + (2\beta-3)r^{-\beta-1} + r^{-\beta}\right)e^{-r}.
\end{eqnarray*}
Hence, it holds for $x\in\mathfrak{R}$, $r=|x|>0$ and $k=1,2$
$$\left|\frac{\partial f_\beta}{\partial x_k}\right|(x) \lesssim \one_{(0,1)}(r)r^{1-\beta}+r^2e^{-r},\quad \left|\frac{\partial^2 f_\beta}{\partial x_1 \partial x_2}\right|(x) \lesssim \one_{(0,1)}(r) r^{-\beta}+r^2e^{-r}.$$
Now, fix any $0<\delta<1$. An application of Proposition \ref{FourierProp} to the function $f_\beta\cdot\one_{[\delta,\infty)^2}$ under consideration of the limit $\delta\downarrow 0$ finally yields  
$$\mathcal{F}((x_1^4+x_2^4)\nu_\beta)(u) \lesssim \frac{1}{|u_1||u_2|},\quad u\in (\mathds{R}^\ast)^2.$$
This proves together with Lemma \ref{trivLevLem} (i) and the continuity of $u\mapsto \mathcal{F}((x_1^4+x_2^4)\nu)(u)$ the Assumption \ref{LevyAssum} (iii) and the first step is accomplished.

\vspace{0,2cm}
{\sc step 2.}
The case $\beta=2$. Let $\phi\,:\mathds{R}_+^2\to [0,1]$ be a $\mathcal{C}^\infty$ function with
\begin{equation*}
\phi(x) =
\begin{cases}
1, & |x|\leq \frac{1}{2} \\
0, & |x|>\frac{3}{4}
\end{cases}, \quad x\in\mathds{R}_+^2.
\end{equation*}
A detailed construction of such a function is given for instance in Rudin \cite{Rud}[\S 1.46]. Set
$$f_2(x)\gl \phi(r)(\log r)^{-2} + (1-\phi(r))e^{-r},\quad r=|x|,\quad x\in\mathfrak{R}$$
and observe that 
$$\nu_2(dx)\gl \one_{\mathfrak{R}}(x) (x_1^4+x_2^4)^{-1} f_2(x)\,\lambda^2(dx)$$
is a L\'{e}vy measure of BGi $2$. Note further that the Assumptions \ref{LevyAssum} (i), (ii) and (iv) hold because of Proposition \ref{BGProp} and
$$|x|^2\lesssim (\log|x|)^{-2} + e^{-|x|}(1-\phi(|x|))\lesssim (\log|x|)^{-2},\quad x\in\mathds{R}_+^2\,:\,|x|\leq \frac{1}{2}.$$ 
Next, we establish the Assumption \ref{LevyAssum} (iii). For this purpose, set $Lr \gl r \log r$, $r>0$ and note that it holds for $0<r<\frac{1}{2}$ and $k=1,2$
\begin{eqnarray*}
\frac{\partial f_2}{\partial x_k}(x) &=& x_k(\phi'(r)(Lr)^{-2}r-2\phi(r)(Lr)^{-3}r),\quad k=1,2, \\
\frac{\partial^2 f_2}{\partial x_1\partial x_2}(x) &=& x_1x_2(\phi''(r)(Lr)^{-2}-4\phi'(r)(Lr)^{-3}-\phi'(r)(Lr)^{-2}r^{-1}+6\phi(r)(Lr)^{-4} \\
&& \hspace{1cm} +4\phi(r)(Lr)^{-3}r^{-1}).
\end{eqnarray*}
This implies for $0<r<\frac{1}{2}$ the asymptotics
$$\left|\frac{\partial f_2}{\partial x_k}\right|(x) \lesssim r^{-1}(\log r)^{-2},\quad \left|\frac{\partial^2 f_2}{\partial x_1\partial x_2}\right|(x) \lesssim r^{-2}(\log r)^{-2}.$$
Observe that we have in the complementary case $r>1$
\begin{eqnarray*}
\frac{\partial f_2}{\partial x_k}(x) &=& -x_kr^{-1}e^{-r}, \\
\frac{\partial^2 f_2}{\partial x_1\partial x_2}(x) &=& x_1x_2(r^{-2} + r^{-3})e^{-r}.
\end{eqnarray*}
This yields for $r>1$ the asymptotics
$$\left|\frac{\partial f_2}{\partial x_k}\right|(x) \lesssim e^{-r},\quad \left|\frac{\partial^2 f_2}{\partial x_1 \partial x_2}\right|(x) \lesssim e^{-r}.$$
Now, we get the claim of Assumption \ref{LevyAssum} (iii) with the same procedure as in the first step. 
\end{proof}

\subsection{Estimating the L\'{e}vy measure}
Denote with $(K_h\lambda^2)\ast((x_1^4+x_2^4)\nu)$ in the following the convolution of the two finite Borel measures $d(K_h\lambda^2)\gl K_h\,d\lambda^2$ and $d((x_1^4+x_2^4)\nu)\gl (x_1^4+x_2^4)\,d\nu$.

\begin{lem} \label{convErrorLem}
Let the above Assumptions \ref{kernelAssum} and \ref{LevyAssum} (iii) hold. Then we have
$$\left|\nu([a,\infty)\times [b,\infty))-
\int_{\mathds{R}^2}g_{a,b}(x)[(K_h\lambda^2)\ast(x_1^4+x_2^4)\nu)](dx)\right|
\lesssim |h\log h| (|(a,b)|^{-2}\vee |(a,b)|^{-4})$$
for all $(a,b)\in\mathfrak{R}$ and $0<h<\frac{1}{2}$.
\end{lem}
\begin{proof}
Write

\begin{eqnarray}
&&\left|\nu([a,\infty)\times [b,\infty))-\int_{\mathds{R}^2} g_{a,b}(x)[(K_h\lambda^2)\ast((x_1^4+x_2^4)\nu)](dx)\right| \nonumber\\ &=&
\left|\int_{\mathds{R}^2}g_{a,b}(x)\left\{[(x_1^4+x_2^4)\nu](dx)-[(K_h\lambda^2)\ast((x_1^4+x_2^4)\nu)](dx)\right\}\right|\nonumber \\&=&
\frac{1}{4\pi} \left|\int_{\mathds{R}^2}(\mathcal{F}g_{a,b})(-u)(1-(\mathcal{F}K_h)(u))\mathcal{F}((x_1^4+x_2^4)\nu)(u)\,\lambda^2(du)\right| \label{convErrorEq}
\end{eqnarray}
where we use for the last inequality the Plancherel identity and the fact that a convolution becomes a simple multiplication in the Fourier space for the last inequality. Note further
$$|1-(\mathcal{F}K_h)(u)| = |(\mathcal{F}K)(0)-(\mathcal{F}K)(hu)| \lesssim \min(h|u|,1),\quad u\in\mathds{R}^2,$$
due to the Lipschitz continuity and our assumption that $K$ is normalized. Hence, using Corollary \ref{gabCor} together with Lemma \ref{trivLevLem} (i), (\ref{convErrorEq}) is up to a constant not larger than $I_1+I_2$ with
$$I_1\gl h|(a,b)|^{-2}\int_{[-1,1]^2}|u|(1+|u_1|)^{-1}(1+|u_2|)^{-1}\,\lambda^2(du)$$
and
$$I_2\gl|(a,b)|^{-4}\int_{\mathds{R}^2}\min(h|u|,1) (1+|u_1|)^{-2}(1+|u_2|)^{-2}\,\lambda^2(du).$$
This finally proves under consideration of Lemma \ref{trivLevLem} (ii), Fubinis theorem and
$$\min(h|u|,1) \leq \min(h|u_1|,1) + \min(h|u_2|,1),\quad \int_\mathds{R} (1+|z|)^{-2}\,\lambda^1(dz)<\infty$$
this lemma.
\end{proof}

\begin{lem} \label{InfDecayRateLem}
Let $\varphi_{\Sigma,\nu,\alpha}$ be the characteristic function of an infinitesimal divisible two dimensional distribution with L\'{e}vy triplet $(\Sigma,\nu,\alpha)$ and finite second moment. Then we have
\begin{equation} \label{expZeroEq}
|\varphi_{\Sigma,\nu,\alpha}(u)|\geq e^{-C(1+|u|)^2},\quad u\in\mathds{R}^2
\end{equation}
with a constant $C$ depending only on the triplet $(\Sigma,\nu,\alpha)$.
\end{lem}

\begin{rem} \rm
Note, that the fast exponential decay to zero in (\ref{expZeroEq}) as $|u|$ tends to infinity results from a possible non-vanishing $\Sigma$. Otherwise $\varphi_{\Sigma,\nu,\alpha}(u)$ may possibly have a slower convergence rate to zero. In this context, review the results in Neumann, Rei{\ss} \cite{NeuRei}. In the case of a compound Poisson process, it is even bounded away from zero, cf. Lemma \ref{CPPDecayLem}.
\end{rem}

\begin{proof}[Proof of Lemma \ref{InfDecayRateLem}]
We have $\varphi(u) = \exp(\Psi(u))$ with
\begin{equation} \label{PsiEq}
\Psi(u) = -\frac{1}{2}\left<u,\Sigma u\right> + i\left<u,\alpha\right> + \int_{\mathds{R}^2}\left (e^{i\left<u,x\right>}-1-i\left<u,x\right>\right)\,\nu(dx),\quad u\in\mathds{R}^2.
\end{equation}
Next, we estimate each summand of $\Psi$ separately:
\begin{eqnarray*}
|\left<u,\Sigma u\right>| &\leq& |u||\Sigma u| \leq |\Sigma||u|^2, \\
|\left<u,\alpha\right>| &\leq& |\alpha||u|,\quad u\in\mathds{R}^2.
\end{eqnarray*}
Furthermore, Sato \cite{Sat}[Lemma 8.6.] yields
$$e^{i\left<u,x\right>} = 1+i\left<u,x\right> + \theta_{u,x}\frac{|\left<u,x\right>|^2}{2}, \quad u,x\in\mathds{R}^2,\quad \theta_{u,x}\in\mathds{C},\quad |\theta_{u,x}|\leq 1$$
which implies
$$\left| \int_{\mathds{R}^2} (e^{i\left<u,x\right>}-1-i\left<u,x\right>)\,\nu(dx)\right|\leq \int_{\mathds{R}^2} |\left<u,x\right>|^2\,\nu(dx)\leq |u|^2\int_{\mathds{R}^2}|x|^2\,\nu(dx).$$
This yields with $C\gl |\alpha|+|\Sigma|+\int_{\mathds{R}^2}|x|^2\,\nu(dx)$ the estimate
$$|\varphi(u)|=|e^{\Psi(u)}| \geq e^{-|\Psi(u)|} \geq e^{-C(1+|u|)^2},\quad u\in\mathds{R}^2$$
and this Lemma is proven.
\end{proof}
In the following we construct, based on low frequency observations, a uniform estimator for the values
$$\{\nu([a,\infty)\times [b,\infty))\,:\,(a,b)\in \mathfrak{R}\}.$$
Let us assume that the corresponding L\'{e}vy process has a finite fourth moment. Our motivation is the following fact:
\begin{eqnarray}
\left(\frac{\partial^4 \Psi_{\Sigma,\nu,\alpha}}{\partial u_1^4} + \frac{\partial^4\Psi_{\Sigma,\nu,\alpha}}{\partial u_2^4}\right)(u) &=& \int_{\mathds{R}^2} e^{i\left<u,x\right>}(x_1^4+x_2^4)\,\nu(dx) \nonumber\\ &=&
\mathcal{F}((x_1^4+x_2^4)\nu)(u),\quad u\in\mathds{R}^2. \label{fourierMotEq}
\end{eqnarray}

\begin{rem} \rm
Note that we have to take at least the third derivations of $\Psi$ in order to remove the Brownian motion part. However, this seems not to be sufficient to deal with L\'{e}vy measures with Blumenthal Getoor indices greater than one. That is why we take the fourth derivations of $\Psi$. Doing so, we are, for example, capable to prove Corollary \ref{BGCor}.
\end{rem}

A simple calculation yields for $k=1,2$
\begin{eqnarray}
\frac{\partial^4 \Psi_{\Sigma,\nu,\alpha}}{\partial u_k^4} &=& 
\frac{\partial^4\varphi_{\Sigma,\nu,\alpha}}{\partial u_k^4} \varphi_{\Sigma,\nu,\alpha}^{-1} 
- 4 \frac{\partial\varphi_{\Sigma,\nu,\alpha}}{\partial u_k} \frac{\partial^3\varphi_{\Sigma,\nu,\alpha}}{\partial u_k^3}\varphi_{\Sigma,\nu,\alpha}^{-2}
- 3 \left(\frac{\partial^2 \varphi_{\Sigma,\nu,\alpha}}{\partial u_k^2}\varphi_{\Sigma,\nu,\alpha}^{-1}\right)^2 \label{deriPsiEq}\\ &&
+ 12 \left(\frac{\partial\varphi_{\Sigma,\nu,\alpha}}{\partial u_k}\right)^2\frac{\partial^2\varphi_{\Sigma,\nu,\alpha}}{\partial u_k^2} \varphi_{\Sigma,\nu,\alpha}^{-3} 
-6\left(\frac{\partial \varphi_{\Sigma,\nu,\alpha}}{\partial u_k} \varphi_{\Sigma,\nu,\alpha}^{-1}\right)^4. \nonumber
\end{eqnarray}
Note that we are going to estimate $\varphi_{\Sigma,\nu,\alpha}$ by
$$\widehat{\varphi}_n(u) = \frac{1}{n}\sum_{t=1}^n e^{i\left<u,X_t-X_{t-1}\right>},\quad u\in\mathds{R}^2.$$
Hence, we set for $k=1,2$
\begin{eqnarray}
\frac{\partial\widehat{\Psi}_n}{\partial u_k^4} &\gl& 
\frac{\partial^4\widehat{\varphi}_n}{\partial u_k^4} \widehat{\varphi}_n^{-1} 
- 4 \frac{\partial\widehat{\varphi}_n}{\partial u_k} \frac{\partial^3\widehat{\varphi}_n}{\partial u_k^3}\widehat{\varphi}_n^{-2}
- 3 \left(\frac{\partial^2 \widehat{\varphi}_n}{\partial u_k^2}\widehat{\varphi}_n^{-1}\right)^2 
+ 12 \left(\frac{\partial\widehat{\varphi}_n}{\partial u_k}\right)^2\frac{\partial^2\widehat{\varphi}_n}{\partial u_k^2} \widehat{\varphi}_n^{-3} \label{deriPsiEqHead}\\ &&
-6\left(\frac{\partial \widehat{\varphi}_n}{\partial u_k} \widehat{\varphi}_n^{-1}\right)^4, \nonumber
\end{eqnarray}
i.e. $\frac{\partial^4\widehat{\Psi}_n}{\partial u_k^4}$ is a function of derivatives of $\widehat{\varphi}_n$ in exactly the same manner as $\frac{\partial^4\Psi_{\Sigma,\nu,\alpha}}{\partial u_k^4}$ is a function of the derivatives of $\varphi_{\Sigma,\nu,\alpha}$, compare (\ref{deriPsiEq}). Of course we cannot write $\widehat{\varphi}_n(u) = e^{\widehat{\Psi}_n(u)}$ since $\widehat{\varphi}_n$ need not be a characteristic function of an infinitesimal divisible measure for each $\omega\in\Omega$. 

Considering (\ref{fourierMotEq}), we set
\begin{equation}\label{NnEq}
\widetilde{N}_n(a,b)\gl\int_{\mathds{R}^2} g_{a,b}(x)\mathcal{F}^{-1}\left(\left(\frac{\partial^4\widehat{\Psi}_n}{\partial u_1^4} + \frac{\partial^4\widehat{\Psi}_n}{\partial u_2^4}\right)\mathcal{F} K_h\right)(x)\,\lambda^2(dx),\quad (a,b)\in\mathfrak{R}
\end{equation}
for an estimator of $\nu([a,\infty)\times[b,\infty))$. Note furthermore that (\ref{NnEq}) is only well-defined on
$$\widetilde{A}_{h,n} \gl \left\{\omega\in\Omega\,:\,\widehat{\varphi}_n(u)\neq 0,\quad\text{for all } u\in\left[-\frac{1}{h},\frac{1}{h}\right]^2\right\},\quad h>0,\quad n\in\mathds{N}$$
because of $\text{supp}(\mathcal{F}K_h)\subseteq \left[-\frac{1}{h},\frac{1}{h}\right]^2$ and $\widehat{\varphi}_n$ has to be non-zero on $\text{supp}(\mathcal{F}K_h)$. At the same time, we have for $\omega\in\widetilde{A}_{h,n}$
$$\left(\frac{\partial^4\widehat{\Psi}_n}{\partial u_1^4} + \frac{\partial^4 \widehat{\Psi}_n}{\partial u_2^4}\right)(\omega)\cdot\mathcal{F}K_h \in L^1(\mathds{R}^2)\cap L^2(\mathds{R}^2)$$
since the left-hand side is a continuous function with compact support. Thus, the inverse Fourier transform in (\ref{NnEq}) is well-defined on $\widetilde{A}_{h,n}$. Set
$$A_{h,n}\gl \left\{|\widehat{\varphi}_n(u)|>\frac{1}{2}|\varphi(u)|,\quad u\in\left[-\frac{1}{h},\frac{1}{h}\right]^2\right\}\subseteq \widetilde{A}_{h,n},\quad h>0,\quad n\in\mathds{N}.$$
Based on the above discussion, finally set
\begin{equation*}
\widehat{N}_n(a,b)\gl
\begin{cases}
\int_{\mathds{R}^2}g_{a,b}(x)\mathcal{F}^{-1}\left(\left(\frac{\partial^4\widehat{\Psi}_n}{\partial u_1^4} + \frac{\partial^4\widehat{\Psi}_n}{\partial u_2^4}\right)\mathcal{F}K_h\right)(x)\,\lambda^2(dx), & \omega\in A_{h,n}, \\
0, & \omega\in A_{h,n}^c,
\end{cases}
\end{equation*}
for all $(a,b)\in\mathfrak{R},\,h>0,\,n\in\mathds{N}.$ Of course, the bandwidth $h=h_n$ has to be chosen in an optimal manner. It turns out that
$$h_n\gl \frac{\log\log n}{\sqrt{\log n}},\quad n\in\mathds{N}$$
yields a satisfying result:
\begin{theorem} \label{LevyEstTh}
It holds under the Assumptions \ref{kernelAssum} and \ref{LevyAssum} (i)-(iii) the asymptotic
\begin{equation} \label{NnEstEq}
\sup_{(a,b)\in\mathfrak{R}} |(a,b)|^2\wedge  |(a,b)|^4 \left|\nu([a,\infty)\times [b,\infty)) - \widehat{N}_n(a,b)\right| = O_{P_{\Sigma,\nu,\alpha}}\left(\frac{(\log\log n)^2}{\sqrt{\log n}}\right),\quad n\to\infty.
\end{equation}
\end{theorem}
\begin{proof}
The proof is divided into three steps. The probability that the inverse Fourier transform is well defined tends to one. This is shown in the first step. The second step estimates the approximation error between $\widehat{\Psi}_n$ and $\Psi$. Finally, the third step uses these estimations together with the statement of Lemma \ref{convErrorLem} to prove the desired convergence rate.

\vspace{0,2cm}
{\sc step 1.}
First, we establish $P(A_n^c)\to 0,\,n\to\infty$ with $A_n\gl A_{h_n,n},\,n\in\mathds{N}$ and $P\gl P_{\Sigma,\nu,\alpha}$. Note for this that we have with $B_{\frac{1}{h}}\gl \left[-\frac{1}{h},\frac{1}{h}\right]^2$ and $\varphi\gl \varphi_{\Sigma,\nu,\alpha}$ the inclusions
\begin{eqnarray*}
A_{h,n}^c = \left\{\exists u\in B_{\frac{1}{h}}\,:\,|\widehat{\varphi}_n(u)|\leq\frac{1}{2}|\varphi(u)|\right\} &\subseteq& \left\{\exists u\in B_{\frac{1}{h}}\,:\,\frac{|\widehat{\varphi}_n(u)-\varphi(u)|}{|\varphi(u)|} \geq \frac{1}{2}\right\} \\&\subseteq&
\left\{\exists u\in B_{\frac{1}{h}}\,:\,\frac{d(\widehat{\varphi}_n,\varphi)}{|\varphi(u)||w(u)|} \geq\frac{1}{2}\right\}.
\end{eqnarray*}
Observe
$$w(u)=(\log(e+|u|))^{-\frac{1}{2}-\delta} \geq e^{-\left(\frac{1}{2}+\delta\right)|u|},\quad u\in\mathds{R}^2,$$
so that together with Lemma \ref{InfDecayRateLem} we obtain
\begin{equation} \label{wphiEstEq}
|\varphi(u)||w(u)|\geq e^{-C(1+|u|)^2},\quad u\in\mathds{R}^2
\end{equation}
with a constant $C>0$. This, finally, implies
$$\left\{\exists u\in B_{\frac{1}{h}}\,:\,\frac{d(\widehat{\varphi}_n,\varphi)}{|\varphi(u)||w(u)|} \geq \frac{1}{2}\right\}\subseteq \left\{d(\widehat{\varphi}_n,\varphi)e^{C\left(1+\frac{\sqrt{2}}{h}\right)^2}\geq\frac{1}{2}\right\},$$
and, the Markov inequality yields together with Corollary \ref{Eboundcor}
\begin{equation} \label{PAnEq}
P(A_{h,n}^c)\leq n^{-\frac{1}{2}}e^{C\left(1+\frac{\sqrt{2}}{h}\right)^2}O(1) \lesssim n^{-\frac{1}{2}}e^{\frac{4C}{h^2}} = e^{-\frac{1}{2} \log n + \frac{4C}{h^2}},\quad n\in\mathds{N}.
\end{equation}
A substitution with $h_n = \frac{\log\log n}{\sqrt{\log n}}$ yields
$$-\frac{1}{2}\log n + \frac{4C}{h_n^2} = - \frac{1}{2}\log n + 4C\frac{\log n}{(\log\log n)^2} \to -\infty,\quad n\to\infty,$$
so that (\ref{PAnEq}) implies $P(A_n^c)\to 0,\,n\to\infty$.

\vspace{0,2cm}
{\sc step 2.}
Next, we consider the difference
$$\frac{\partial^4\widehat{\Psi}_n}{\partial u_k^4}-\frac{\partial^4\Psi}{\partial u_k^4},\quad k=1,2,\quad n\in\mathds{N}.$$
(\ref{deriPsiEq}) and (\ref{deriPsiEqHead}) consist of respectively five terms. Subtracting (\ref{deriPsiEq}) from (\ref{deriPsiEqHead}) results in five difference terms. We rearrange for $k=1,2$ these terms in (\ref{term1Eq}) - (\ref{term5Eq}) for our needs:
\begin{eqnarray}
&&\frac{\partial^l\varphi}{\partial u_k^l}\varphi^{-1} - \frac{\partial^l\widehat{\varphi}_n}{\partial u_k^l} \widehat{\varphi}_n^{-1},\quad l=1,2,3,4 \label{term1Eq}\\ 
&=& \frac{\partial^l(\varphi-\widehat{\varphi}_n)}{\partial u_k^l}\widehat{\varphi}_n^{-1} + \frac{\partial^l \varphi}{\partial u_k^l}\varphi^{-1}(\widehat{\varphi}_n-\varphi)\widehat{\varphi}_n^{-1}, \nonumber \\ \nonumber \\ \nonumber\\
&&\frac{\partial \varphi}{\partial u_k}\frac{\partial^3\varphi}{\partial u_k^3}\varphi^{-2} -
\frac{\partial \widehat{\varphi}_n}{\partial u_k}\frac{\partial^3\widehat{\varphi}_n}{\partial u_k^3}\widehat{\varphi}_n^{-2} \\ &=& 
\frac{\partial\varphi}{\partial u_k}\varphi^{-1}\left(\frac{\partial^3\varphi}{\partial u_k^3}\varphi^{-1}-\frac{\partial^3\widehat{\varphi}_n}{\partial u_k^3}\widehat{\varphi}_n^{-1}\right) +\frac{\partial^3\widehat{\varphi}_n}{\partial u_k^3}\widehat{\varphi}_n^{-1}\left(\frac{\partial\varphi}{\partial u_k}\varphi^{-1} - \frac{\widehat{\varphi}_n}{\partial u_k}\widehat{\varphi}_n^{-1}\right), \nonumber \\ \nonumber \\ \nonumber \\
&&\left(\frac{\partial^2 \varphi}{\partial u_k^2}\varphi^{-1}\right)^2 - \left(\frac{\partial^2\widehat{\varphi}_n}{\partial u_k^2} \widehat{\varphi}_n^{-1}\right)^2 \\&=& 
\left(\frac{\partial^2\varphi}{\partial u_k^2}\varphi^{-1} + \frac{\partial^2\widehat{\varphi}_n}{\partial u_k^2}\widehat{\varphi}_n^{-1}\right)\left(\frac{\partial^2\varphi}{\partial u_k^2}\varphi^{-1} - \frac{\partial^2\widehat{\varphi}_n}{\partial u_k^2}\widehat{\varphi}_n^{-1}\right), \nonumber\\ \nonumber \\ \nonumber \\
&&\left(\frac{\partial\varphi}{\partial u_k}\right)^2\frac{\partial^2\varphi}{\partial u_k^2} \varphi^{-3} - \left(\frac{\partial \widehat{\varphi}_n}{\partial u_k}\right)^2\frac{\partial^2 \widehat{\varphi}_n}{\partial u_k^2}\widehat{\varphi}_n^{-3} \\&=&
\left(\frac{\partial \varphi}{\partial u_k}\right)^2\varphi^{-2}\left(\frac{\partial^2\varphi}{\partial u_k^2}\varphi^{-1}-\frac{\partial^2\widehat{\varphi}_n}{\partial u_k^2}\widehat{\varphi}_n^{-1}\right) + \frac{\partial^2\widehat{\varphi}_n}{\partial u_k^2} \widehat{\varphi}_n^{-1}\left(\left(\frac{\partial\varphi}{\partial u_k}\varphi^{-1}\right)^2 - \left(\frac{\partial \widehat{\varphi}_n}{\partial u_k} \widehat{\varphi}_n^{-1}\right)^2\right), \nonumber \\ \nonumber \\ \nonumber \\
&&\left(\frac{\partial \varphi}{\partial u_k}\varphi^{-1}\right)^4 - \left(\frac{\partial\widehat{\varphi}_n}{\partial u_k}\widehat{\varphi}_n^{-1}\right)^4  \label{term5Eq} \\ &=& 
\left(\left(\frac{\partial\varphi}{\partial u_k}\varphi^{-1}\right)^2 + \left(\frac{\partial \widehat{\varphi}_n}{\partial u_k}\widehat{\varphi}_n^{-1}\right)^2\right) \left(\left(\frac{\partial\varphi}{\partial u_k}\varphi^{-1}\right)^2 - \left(\frac{\partial \widehat{\varphi}_n}{\partial u_k}\widehat{\varphi}_n^{-1}\right)^2\right). \nonumber
\end{eqnarray}

Next, after some straightforward calculations, observe
\begin{eqnarray}
\frac{\partial \varphi}{\partial u_k} \varphi^{-1} &=& \frac{\partial \Psi}{\partial u_k}, \label{phiDer1Eq}\\
\frac{\partial^2\varphi}{\partial u_k^2} \varphi^{-1} &=& \left(\frac{\partial \Psi}{\partial u_k}\right)^2 + \frac{\partial^2\Psi}{\partial u_k^2}, \\
\frac{\partial^3 \varphi}{\partial u_k^3} \varphi^{-1} &=& \left(\frac{\partial\Psi}{\partial u_k}\right)^3 + \frac{\partial^3 \Psi}{\partial u_k^3} + 3\frac{\partial\Psi}{\partial u_k}\frac{\partial^2\Psi}{\partial u_k^2},  \\
\frac{\partial^4 \varphi}{\partial u_k^4} \varphi^{-1} &=& \left(\frac{\partial \Psi}{\partial u_k}\right)^4 + 3\left(\frac{\partial^2 \Psi}{\partial u_k^2}\right)^2 + 6\left(\frac{\partial \Psi}{\partial u_k}\right)^2\frac{\partial^2\Psi}{\partial u_k^2} + 4\frac{\partial \Psi}{\partial u_k}\frac{\partial^3 \Psi}{\partial u_k^3} + \frac{\partial^4 \Psi}{\partial u_k^4}. \label{phiDer4Eq}
\end{eqnarray}
In what follows, we estimate the derivatives of $\Psi$:
$$\left<u,\Sigma u\right> = \sigma_{11}u_1^2+2\sigma_{12}u_1u_2+\sigma_{22}u_2^2,\quad u\in\mathds{R}^2$$
and the representation (\ref{PsiEq}) yields
$$\frac{\partial\Psi}{\partial u_1}(u) = -\frac{1}{2}(2u_1\sigma_{11}+2u_2\sigma_{12})+i\alpha_1+\int_{\mathds{R}^2}\left(ix_1e^{i\left<u,x\right>}-ix_1\right)\,\nu(dx)$$
which yields together with 
$$\left|e^{i\left<u,x\right>}-1\right| \leq |\left<u,x\right>|\leq |u||x|,\quad u,x\in\mathds{R}^2$$
and $\int_{\mathds{R}^2} |x|^2\,\nu(dx)<\infty$ the inequality
$$\left|\frac{\partial \Psi}{\partial u_1}\right|(u)\lesssim 1+|u|,\quad u\in\mathds{R}^2$$
where the constant in the $\lesssim$ sign depends only on the L\'{e}vy triplet $(\Sigma,\nu,\alpha)$. Similarly we get
$$\left|\frac{\partial^2\Psi}{\partial u_1^2}\right|(u) = \left|-\sigma_{11}-\int_{\mathds{R}^2}x_1^2e^{i\left<u,x\right>}\,\nu(dx)\right|\leq \sigma_{11}+\int_{\mathds{R}^2}x_1^2\,\nu(dx) < \infty$$
and
\begin{eqnarray*}
\left|\frac{\partial^3 \Psi}{\partial u_1^3}\right|(u) &=& \left|-i\int_{\mathds{R}^2}x_1^3e^{i\left<u,x\right>}\,\nu(dx)\right|\leq\int_{\mathds{R}^2}|x_1|^3\,\nu(dx) <\infty, \\
\left|\frac{\partial^4 \Psi}{\partial u_1^4}\right|(u) &=& \left|\int_{\mathds{R}^2}x_1^4e^{i\left<u,x\right>}\,\nu(dx)\right|\leq\int_{\mathds{R}^2}x_1^4\,\nu(dx) <\infty.
\end{eqnarray*}
The derivatives $\frac{\partial^l}{\partial u_2^l},\,l=1,2,3,4$ yield analogous estimates. Hence, (\ref{phiDer1Eq}) - (\ref{phiDer4Eq}) imply
\begin{equation} \label{DphiEq}
\left|\frac{\partial^l \varphi}{\partial u_k^l} \varphi^{-1}\right|(u) \lesssim (1+|u|)^l,\quad u\in\mathds{R}^2,\quad l=1,2,3,4.
\end{equation}
Additionally (\ref{DphiEq}) yields together with (\ref{term1Eq}) on $A_n$
\begin{eqnarray}
\left|\frac{\partial^l \widehat{\varphi}_n}{\partial u_k^l} \widehat{\varphi}_n^{-1}\right|(u) &\leq& \left|\frac{\partial^l\varphi}{\partial u_k^l} \varphi^{-1}\right|(u) + \left|\frac{\partial^l\varphi}{\partial u_k^l}\varphi^{-1} - \frac{\partial^l\widehat{\varphi}_n}{\partial u_k^l}\widehat{\varphi}_n^{-1}\right|(u) \label{DphiHeadEq} \\&\lesssim&
\left(1+\frac{d(\varphi,\widehat{\varphi}_n)}{w(u)|\varphi(u)|}\right)(1+|u|)^l,\quad u\in\left[-\frac{1}{h_n},\frac{1}{h_n}\right]^2, \quad l=1,2,3,4. \nonumber
\end{eqnarray}
Hence, (\ref{deriPsiEq}), (\ref{deriPsiEqHead}) and (\ref{term1Eq}) - (\ref{term5Eq}) and (\ref{DphiEq}), (\ref{DphiHeadEq}) finally yield on $A_n$
\begin{equation*}
\left|\frac{\partial^4\widehat{\Psi}_n}{\partial u_k^4} - \frac{\partial^4 \Psi}{\partial u_k^4}\right|(u) \lesssim \sum_{j=1}^4\left(\frac{d(\varphi,\widehat{\varphi}_n)}{w(u)|\varphi(u)|}\right)^j(1+|u|)^4, \quad u\in\left[-\frac{1}{h_n},\frac{1}{h_n}\right]^2.
\end{equation*}

\vspace{0,2cm}
{\sc step 3.}
Next, observe with Sato \cite{Sat}[Proposition 2.5 (xii)]
$$\mathcal{F}^{-1}\left(\left(\frac{\partial^4\Psi}{\partial u_1^4} + \frac{\partial^4\Psi}{\partial u_2^4}\right)\mathcal{F} K_{h_n}\right)\lambda^2 = (K_{h_n}\lambda^2)\ast((x_1^4+x_2^4)\nu).$$
Using the Plancherel identity, we get the following essential estimates:
\begin{eqnarray}
&&\one_{A_n}\left|\int_{\mathds{R}^2} g_{a,b}(x)\mathcal{F}^{-1}\left(\left(\frac{\partial^4\widehat{\Psi}_n}{\partial u_1^4} + \frac{\partial^4 \widehat{\Psi}_n}{\partial u_2^4} \right) \mathcal{F} K_{h_n}\right)(x)\,\lambda^2(dx) \right.\nonumber\\ &&
\left.-\int_{\mathds{R}^2} g_{a,b}(x)\mathcal{F}^{-1}\left(\left(\frac{\partial^4\Psi}{\partial u_1^4} + \frac{\partial^4 \Psi}{\partial u_2^4} \right) \mathcal{F} K_{h_n}\right)(x)\,\lambda^2(dx)\right| \nonumber\\&=&
\one_{A_n}\frac{1}{4\pi^2} \left|\int_{\mathds{R}^2} (\mathcal{F}g_{a,b})(-u)\sum_{k=1}^2\frac{\partial^4(\widehat{\Psi}_n-\Psi)}{\partial u_k^4}(u)\mathcal{F}K_{h_n}(u)\,\lambda^2(du)\right| \nonumber\\ &\lesssim&
\one_{A_n} |(a,b)|^{-2} \int_{\left[-\frac{1}{h_n},\frac{1}{h_n}\right]^2}\sum_{j=1}^4\left(\frac{d(\varphi,\widehat{\varphi}_n)}{w(u)|\varphi(u)|}\right)^j(1+|u|)^4\,\lambda^2(du). \label{nuEstLastEq}
\end{eqnarray}
Note that
$$d(\varphi,\widehat{\varphi}_n) = O_P(n^{-\frac{1}{2}})$$
and
$$(w(u)|\varphi(u)|)^{-j}(1+|u|)^4\lesssim e^{C(1+|u|)^2},\quad j=1,2,3,4,\quad u\in\mathds{R}^2$$
hold for suitable $C>0$, compare (\ref{wphiEstEq}). Hence, (\ref{nuEstLastEq}) is not larger than

$$|(a,b)|^{-2}\int_{\left[-\frac{1}{h_n},\frac{1}{h_n}\right]^2}e^{C(1+|u|)^2}\,\lambda^2(du)\cdot O_P(n^{-\frac{1}{2}}) = |(a,b)|^{-2}O_P(n^{\epsilon-\frac{1}{2}}),\quad n\to\infty$$
for every $\epsilon>0$. Together with Lemma \ref{convErrorLem} and
$$|h_n\log h_n| = -\frac{\log\log n}{\sqrt{\log n}}(\log\log\log n - \frac{1}{2}\log\log n) \lesssim \frac{(\log\log n)^2}{\sqrt{\log n}}$$
this proves this theorem.
\end{proof}

\subsection{The inverting operation}
Considering the L\'{e}vy copula (\ref{introLevCopEq}), our next goal is to establish an inversion operation. For this purpose, we first define some function spaces and an inversion operation $\mathcal{I}$ on those spaces: Set
\begin{eqnarray*}
\widehat{\mathcal{C}}&\gl& \{g\,:\,(0,\infty)\to\mathds{R}_+,\quad g\in\mathcal{C},\quad\lim_{x\to\infty} g(x)=0\}, \\
\mathcal{D}_\delta &\gl& \{h\,:\, (0,\infty)\to\mathds{R}_+, \quad h \text{ is c\`adl\`ag, decreasing and} \lim_{x\to\infty} h(x) = \delta\},\quad \delta>0, \\
\mathcal{D} &\gl& \bigcup_{\delta>0} \mathcal{D}_\delta 
\end{eqnarray*}
and
\begin{equation} \label{defIEq}
\begin{array}{llcl}
\mathcal{I} \,:\,&\widehat{\mathcal{C}} \times (0,\infty) &\to& \mathcal{D} \\
&(g,\delta) &\mapsto& (z \mapsto \inf\{x\geq \delta\,:\,\inf_{\delta\leq y\leq x} g(y) \leq z\}).
\end{array}
\end{equation}
Furthermore, let $R\,:\,(0,\infty) \to (0,\infty)$ be a function and $(\epsilon_n)$, $(\delta_n)$ be two sequences of positive numbers, such that
$$\gamma_n\gl R(\delta_n)\epsilon_n \to 0,\quad \epsilon_n\downarrow 0,\quad \delta_n\downarrow 0,\quad n\to\infty$$
hold. Note that we have
$$\mathcal{I}(g,\delta)\in\mathcal{D}_\delta,\quad g\in\widehat{C},\quad\delta>0.$$

The next Proposition \ref{invProp} investigates the behavior of an approximation error under the inversion operation $\mathcal{I}$. Note that $\mathcal{I}$ is the pseudo inverse with the starting position $\delta>0$, cf. (\ref{defIEq}). The introduction of such an offset $\delta>0$ is required for the subsequent treatment of the small jumps. Note that we have $\Lambda=\infty$ ($\Lambda$ in Proposition \ref{invProp}) in this section. The case $\Lambda<\infty$ is important for the investigations of the next Section \ref{CPPSec}.

\begin{prop} \label{invProp}
Let $f\,:\,(0,\infty) \to (0,\Lambda)$, $\Lambda\in (0,\infty]$ be a $\mathcal{C}^1$-bijection with 
$$f'<0,\quad \inf_{0<x\leq 1} |f'(x)|>0$$ 
and let $(f_n)_n\subseteq \widehat{\mathcal{C}}$ be a family of functions, such that
$$\sup_{x\geq \delta_n} |f_n(x)-f(x)|\leq \gamma_n,\quad n\in\mathds{N}$$
holds. Fix any $0<a<b<\Lambda$. Then it holds also for each $n\in\mathds{N}$ with $2\gamma_n < a \wedge (\Lambda - b)$ and $\delta_n< f^{-1}(b+2\gamma_n)$ the inequality
$$\sup_{a\leq z\leq b} |\mathcal{I}(f_n,\delta_n)(z) - f^{-1}(z)|\leq 2\gamma_n\left(\inf_{0< x\leq f^{-1}\left(\frac{a}{2}\right)} |f'(x)|\right)^{-1}.$$
\end{prop}
\begin{proof}
Set
$$F_n(x)\gl\inf_{\delta_n\leq y\leq x}f_n(y),\quad h_n(z)\gl\mathcal{I}(f_n,\delta_n)(z),\quad x\geq\delta_n,\quad z>0.$$
Note that $F_n\,:\,[\delta_n,\infty)\to\mathds{R}_+$ is a decreasing, continuous function with $F_n(x)\to 0$, $x\to\infty$ for each $n\in\mathds{N}$. First, we show the inequality
\begin{equation} \label{FIneq}
\sup_{x\geq \delta_n} |f(x)-F_n(x)| \leq \gamma_n,\quad n\in\mathds{N}.
\end{equation}
Note for this
$$F_n(x)=f_n(c_x)\leq f_n(x),\quad \delta_n\leq c_x\leq x$$
and
\begin{eqnarray*}
f(x) - F_n(x) &=& f(x)-f_n(c_x) \leq f(c_x)-f_n(c_x) \leq \gamma_n, \\
f(x) - F_n(x) &\geq& f(x) - f_n(x) \geq -\gamma_n
\end{eqnarray*}
for all $n\in\mathds{N}$. Next, fix any $0<a\leq z\leq b<\Lambda$ and $n\in\mathds{N}$ with
$2\gamma_n<a \wedge (\Lambda-b)$, $\delta_n < f^{-1}(b+2\gamma_n)$. Set
$$y\gl z-\gamma_n > \frac{a}{2},\quad y'\gl z +2\gamma_n<\Lambda$$
and
$$x\gl f^{-1}(y)\geq x'\gl f^{-1}(y')\geq f^{-1}(b+2\gamma_n)\geq \delta_n.$$
(\ref{FIneq}) implies $F_n(x)\leq f(x)+\gamma_n$ which yields, since $h_n$ is the pseudo-inverse of $F_n$,
$$h_n(z) = h_n(f(x)+\gamma_n) \leq x = f^{-1}(z-\gamma_n).$$
Equally, we have $F_n(x') \geq f(x') - \gamma_n > f(x')-2\gamma_n$ which implies
$$h_n(z) = h_n(f(x')-2\gamma_n)\geq x' = f^{-1}(z+2\gamma_n),$$
so that, altogether we have
$$f^{-1}(z+2\gamma_n)\leq h_n(z) \leq f^{-1}(z-\gamma_n).$$
Using the mean value theorem, this yields, on the one hand,
$$h_n(z)-f^{-1}(z) \leq f^{-1}(z-\gamma_n)-f^{-1}(z) = -\gamma_n(f^{-1})'(\xi_1)$$
and on the other hand
$$f^{-1}(z)-h_n(z)\leq f^{-1}(z)-f^{-1}(z+2\gamma_n) = -2\gamma_n(f^{-1})'(\xi_2)$$
with
$$\xi_1,\xi_2\in [z-\gamma_n,z+2\gamma_n]\subseteq \left[\frac{a}{2},\Lambda\right).$$
Thus, we finally obtain
$$|h_n(z)-f^{-1}(z)|\leq 2\gamma_n\sup_{\frac{a}{2}\leq y<\Lambda}|(f^{-1})'(y)| = 2\gamma_n\left(\inf_{0 < x\leq f^{-1}\left(\frac{a}{2}\right)} |f'(x)|\right)^{-1}.$$
\end{proof}

Next, we state a stochastic version of Proposition \ref{invProp}, which is adapted to our later needs.

\begin{cor} \label{LevCopInvCor}
Given a probability space $(\Omega,\mathcal{F},P)$ and a family of functions
$$\widehat{Z}_n\,:\,\Omega\to \widehat{\mathcal{C}},\quad n\in\mathds{N},$$
such that $\omega\mapsto[\widehat{Z}_n(\omega)](x)$ is $\mathcal{F}$-measurable for every $n\in\mathds{N}$, $x>0$ and such that
\begin{equation} \label{invStochLemEq}
\sup_{x\geq \delta_n} |\widehat{Z}_n(x) - f(x)| = O_P(\gamma_n),\quad n\to\infty
\end{equation}
holds with a function $f$ as in Proposition \ref{invProp}.
Then it also holds for any fixed $0<a<b<\Lambda$
$$\sup_{a\leq z\leq b} |\mathcal{I}(\widehat{Z}_n,\delta_n)(z)-f^{-1}(z)| = O_P(\gamma_n),\quad n\to\infty.$$
\end{cor}
\begin{proof}
Write $(\gamma_n X_n)_{n\in\mathds{N}}$ instead of $O_P(\gamma_n)$ in (\ref{invStochLemEq}), i.e. $(X_n)_n$ is a family of random variables, which are uniformly bounded in probability. Set furthermore 
$$A_n\gl \left\{2\gamma_n X_n < a\wedge (\Lambda-b),\quad\delta_n < f^{-1}(b+2\gamma_n X_n)\right\},\quad n\in\mathds{N}.$$ 
Then, Proposition \ref{invProp} states that we have for $\omega\in A_n$
$$\sup_{a\leq z\leq b} |\mathcal{I}(\widehat{Z}_n(\omega),\delta_n)(z)-f^{-1}(z)| \leq 2\gamma_n X_n(\omega)\left(\inf_{0< x\leq f^{-1}\left(\frac{a}{2}\right)} |f'(x)|\right)^{-1}\lesssim \gamma_n X_n(\omega).$$
This proves this Corollary since $P(A_n^c)\to 0$ for $n\to\infty$.
\end{proof}

Finally, we combine the statements developed so far and get the following main result:

\begin{theorem} \label{LevyGeneralTheorem}
Let the Assumptions \ref{kernelAssum} and \ref{LevyAssum} hold and $0<a<b<\infty$ be two fixed numbers. Set $\delta_n\gl (\log\log n)^{-1}$ and
$$\widehat{U}_{1,n}^{-1} \gl \mathcal{I}(\Re_+\widehat{N}_n(\cdot,0),\delta_n),\quad \widehat{U}_{2,n}^{-1} \gl \mathcal{I} (\Re_+\widehat{N}_n(0,\cdot),\delta_n),\quad n\in\mathds{N}.$$
Then, it holds with the plug-in estimator
$$\widehat{\mathfrak{C}}_n(u,v)=\widehat{N}_n(\widehat{U}_{1,n}^{-1}(u),\widehat{U}_{2,n}^{-1}(v)),\quad u,v>0$$
the asymptotic
$$\sup_{a\leq u,v\leq b}|\mathfrak{C}(u,v)-\widehat{\mathfrak{C}}_n(u,v)| = O_{P_{\Sigma,\nu,\alpha}}\left(\frac{(\log\log n)^9}{\sqrt{\log n}}\right),\quad n\to\infty.$$
\end{theorem}
\begin{proof}
First, note that we can replace $\widehat{N}_n$ by $\Re_+ \widehat{N}_n$ and (\ref{NnEstEq}) is still valid. This is due to the fact that we have for all $c\in\mathds{C}$ and $r\in\mathds{R}_+$ the inequality
$$|c-r|=\sqrt{(\Re(c-r))^2+(\Im(c-r))^2}\geq |\Re(c-r)| = |\Re(c)-r|\geq |\Re_+(c)-r|.$$
Observe furthermore
$$\Re_+\widehat{N}_n(\cdot,0),\,\Re_+\widehat{N}_n(0,\cdot)\in\widehat{\mathcal{C}},\quad n\in\mathds{N}.$$
Set
$$\epsilon_n\gl \frac{(\log \log n)^2}{\sqrt{\log n}}, \quad R(x)\gl x^{-4},\quad x>0$$
and note that
$$\gamma_n = R(\delta_n)\epsilon_n = \frac{(\log\log n)^6}{\sqrt{\log n}} \to 0,\quad n\to\infty.$$
Theorem \ref{LevyEstTh} states
$$\sup_{x\geq \delta_n} |\Re_+\widehat{N}_n(x,0) - U_1(x)| = O_{P_{\Sigma,\nu,\alpha}}(\gamma_n),\quad n\to\infty,$$
so that Corollary \ref{LevCopInvCor} implies with $\Lambda\gl \infty$, $\widehat{Z}_n\gl \Re_+\widehat{N}_n(\cdot,0)$ and Assumption \ref{LevyAssum} (iv)
\begin{equation} \label{U1InvEq}
\sup_{a\leq u\leq b} |\widehat{U}_{1,n}^{-1}(u)-U_1^{-1}(u)| = O_{P_{\Sigma,\nu,\alpha}}(\gamma_n),\quad n\to\infty.
\end{equation}
Of course, exactly the same considerations yield the $U_2$ analogue of (\ref{U1InvEq}). Next, write
$$u_n\gl U_1\circ \widehat{U}_{1,n}^{-1}(u),\quad v_n\gl U_2\circ \widehat{U}_{2,n}^{-1}(v),\quad a\leq u,v\leq b$$
and note that the L\'{e}vy-copula $\mathfrak{C}$ is Lipschitz continuous, cf. Kallsen, Tankov \cite{KalTan}[Lemma 3.2]. More precise, we have
\begin{equation} \label{LevCopEq1}
|\mathfrak{C}(u,v)-\mathfrak{C}(u',v')|\leq |u-u'|+|v-v'|,\quad u,u',v,v'>0.
\end{equation}
Together with the mean value theorem and $\widehat{U}_{1,n}^{-1}\in \mathcal{D}_{\delta_n}$ we have for $a\leq u\leq b$
$$|u-u_n| = |U_1\circ U_1^{-1}(u)-U_1\circ \widehat{U}_{1,n}^{-1}(u)| = |U_1^{-1}(u)-\widehat{U}_{1,n}^{-1}(u)||U_1'(x)|,\quad x\in [U_1^{-1}(b)\wedge \delta_n,\infty).$$
Thus, (\ref{U1InvEq}) yields together with Assumption \ref{LevyAssum} (iv)
\begin{equation} \label{LevCopEq2}
\sup_{a\leq u\leq b} |u-u_n| = \delta_n^{-3} O_{P_{\Sigma,\nu,\alpha}}(\gamma_n),\quad n\to\infty.
\end{equation}
Hence, (\ref{LevCopEq1}) and (\ref{LevCopEq2}) imply
\begin{equation} \label{LevCopEq3}
\sup_{a\leq u,v\leq b} |\mathfrak{C}(u,v)-\mathfrak{C}(u_n,v_n)| \leq \sup_{a\leq u\leq b} |u-u_n| + \sup_{a\leq v\leq b} |v-v_n| = \delta_n^{-3} O_{P_{\Sigma,\nu,\alpha}}(\gamma_n),\quad n\to\infty.
\end{equation}
Theorem \ref{LevyEstTh} yields because of $\widehat{U}_{j,n}^{-1}\in\mathcal{D}_{\delta_n}$, $j=1,2$ the asymptotic
\begin{eqnarray}
\sup_{a\leq u,v\leq b}|\mathfrak{C}(u_n,v_n)-\widehat{\mathfrak{C}}_n(u,v)| &=& \sup_{a\leq u,v\leq b}|U(\widehat{U}_{1,n}^{-1}(u),\widehat{U}_{2,n}^{-1}(v))-\widehat{N}_n(\widehat{U}_{1,n}^{-1}(u),\widehat{U}_{2,n}^{-1}(v))|  \nonumber\\&=&
O_{P_{\Sigma,\nu,\alpha}}(\gamma_n),\quad n\to\infty. \label{LevCopEq4}
\end{eqnarray}
Finally, (\ref{LevCopEq3}) and (\ref{LevCopEq4}) prove this theorem.
\end{proof}

\section{The Compound Poisson Process (CPP) Case} \label{CPPSec}
Note that we do not use the special shape of the weight function
$$w(u) = (\log(e+|u|))^{-\frac{1}{2}-\delta},\quad \delta>0,\quad u\in \mathds{R}^2$$
in the previous Section \ref{LevCopSec}. Neither do we use the convergence rate $\sqrt{n}$ obtained in Theorem \ref{Eboundth}. In fact, the proofs in the previous section also work if we merely had
$$e^{-C(1+|u|)^2} \lesssim w(u),\quad u\in\mathds{R}^2$$
for some constant $C>0$ and, concerning Theorem \ref{Eboundth},
$$n^{-\frac{1}{2}+\epsilon}\sup_{n\geq 1} E_{\Sigma,\nu,\alpha} \left\|\frac{\partial^l}{\partial u_k^l} A_n(u)\right\|_{L^\infty(w)} < \infty,\quad k=1,2,\quad 0\leq l\leq 4$$
for some $\epsilon>0$. This is due to the fast decay behavior of $\varphi_{\Sigma,\nu,\alpha}$ if $\Sigma\neq 0$, cf. Lemma \ref{InfDecayRateLem}. Therefore we cannot derive any benefit from these stronger results. However, if $\varphi_{\Sigma,\nu,\alpha}$ decays more slowly, we can benefit from these stronger results as we shall demonstrate in the case of a compound Poisson process with drift. This is, in some sense, the complementary case of the one we investigated in the previous section.

\begin{assum} \label{CompoundAssum}
We state here the assumptions concerning $\nu$ in the compound Poisson case:
\begin{itemize}
\item[(i)]
The corresponding L\'{e}vy process is a compound Poisson process with intensity $0<\Lambda<\infty$ and has only positive jumps, i.e.
$$0<\Lambda\gl\nu(\mathds{R}^2)=\nu(\mathds{R}_+^2)<\infty,$$ 
\item[(ii)]
$\exists \gamma> 0\,:\,\int |x|^{8+\gamma}\,\nu(dx) < \infty, \quad \text{i.e. finite 8+$\gamma$-th moment},$
\item[(iii)]
$\mathcal{F}((x_1^4+x_2^4)\nu)(u) \lesssim (1+|u_1|)^{-1}(1+|u_2|)^{-1},\quad u\in\mathds{R}^2,$
\item[(iv)]
$U_k\,:\, (0,\infty) \to (0,\Lambda)$ is a $\mathcal{C}^1$-bijection with $U_k'<0$ and
$$\inf_{0<x_k\leq 1}|U_k'(x_k)|>0,\quad \sup_{x_k>0} (1\wedge x_k) |U_k'(x_k)| < \infty,\quad k=1,2.$$
\end{itemize}
\end{assum}

\begin{prop} \label{CPPAssumProp}
Let $f\,:\,\mathds{R}_+^2\to\mathds{R}_+$ be a continuous function with the properties
\begin{itemize}
\item[(i)]
$f(x)>0,\quad x\in(\{0\}\times\mathds{R}_+^\ast)\times(\mathds{R}_+^\ast\times\{0\})$,
\item[(ii)]
$f(x) \lesssim |x|^2(\log|x|)^{-2},\quad x\in\mathds{R}_+^2\,:\,|x|\leq\frac{1}{2}$,
\item[(iii)]
$f(x)\lesssim (1+|x|)^{-(6+\epsilon)},\quad x\in \mathds{R}_+^2$
\end{itemize}
for some $\epsilon>0$. Then 
$$\nu(dx)\gl \one_{\mathfrak{R}}(x)(x_1^4+x_2^4)^{-1}f(x) \lambda^2(dx)$$
is a L\'{e}vy measure and fulfills the Assumptions \ref{CompoundAssum} (i), (ii) and (iv).
\end{prop}

\begin{proof}
We only highlight the deviations from the proof in Proposition \ref{BGProp}:
\begin{itemize}
\item[(i)]
We have
$$0<\Lambda = \nu(\mathds{R}_+^2) \lesssim \int_{\mathds{R}_+^2} |x|^{-4} f(x)\,\lambda^2(dx) \lesssim \int_0^{\frac{1}{2}} r^{-4} r^2 (\log r)^{-2} r\,dr + \int_{\frac{1}{2}}^\infty r^{-4} r\,dr < \infty.$$
\item[(iv)]
Observe first
$$[x_1,\infty)\times\mathds{R}_+ \uparrow \mathds{R}_+^\ast\times\mathds{R}_+,\quad x_1\downarrow 0,$$
so that we obtain
$$U_1(x_1) = \nu([x_1,\infty)\times\mathds{R}_+) \to \nu(\mathds{R}_+^\ast\times\mathds{R}_+) = \nu(\mathds{R}_+^2) = \Lambda.$$
This yields that $U_1\,:\,(0,\infty)\to (0,\Lambda)$ is a surjection. Compare for this the argumentation in Proposition \ref{BGProp}. Finally, note that it holds for $x_1>0$
$$|U_1'(x_1)|=\int_0^\infty (x_1^4+y_2^4)^{-1} f(x_1,y_2)\,dy_2 \lesssim \int_0^1(x_1+y_2)^{-2}\,dy_2 + \int_1^\infty y_2^{-4}\,dy_2 = \frac{1}{x_1} - \frac{1}{x_1+1} + \frac{1}{3}$$
which implies
$$\sup_{x_1>0} (1\wedge x_1) |U_1'(x_1)| <\infty.$$
\end{itemize}
\end{proof}

\begin{cor}
It exists a L\'{e}vy measure $\nu_0$ that fulfills the Assumptions \ref{CompoundAssum} with the property
$$\int_{\mathds{R}^2} |x|^{-\epsilon} \nu_0(dx) = \infty$$
for all $\epsilon>0$.
\end{cor}
\begin{proof}
We imitate the proof of step 2 in Corollary \ref{BGCor}. For this purpose, set
$$f_0(x)\gl \psi(r) (\log r)^{-2} + (1-\phi(r))e^{-r},\quad r\gl |x|,\quad x\in \mathfrak{R}$$
and
$$\nu_0(dx) \gl \one_{\mathfrak{R}}(x) (x_1^4+x_2^4)^{-1} f_0(x)\,\lambda^2(dx)$$
with $\phi$ as in the proof of Corollary \ref{BGCor} and
$$\psi(x)\gl |x|^2\phi(x) = r^2\phi(r),\quad x\in\mathfrak{R}.$$
Then Assumptions \ref{CompoundAssum} (i), (ii) and (iv) are fulfilled because of Proposition \ref{CPPAssumProp}. Furthermore, we have
$$\int_{\mathds{R}^2} |x|^{-\epsilon}\,\nu_0(dx) \geq \int_0^{\frac{1}{2}} r^{-\epsilon} r^{-4} r^2 (\log r)^{-2} r\,dr \gtrsim \int_0^{\frac{1}{2}} r^{-1-\frac{\epsilon}{2}}\,dr = \infty$$
for all $\epsilon>0$. Concerning Assumption \ref{CompoundAssum} (iii), note that it holds for $r\geq 0$
\begin{eqnarray*}
\psi'(r) &=& 2r\phi(r) + r^2\phi'(r), \\
\psi''(r) &=& 2\phi(r) + 4r\phi'(r) + r^2\phi''(r),
\end{eqnarray*}
i.e. $\|\psi\|_{\infty}<\infty$, $\|\psi'\|_{\infty}<\infty$ and $\|\psi''\|_{\infty}<\infty$. The remaining proof works exactly as step 2 in the proof of Corollary \ref{BGCor}.
\end{proof}

Note that Assumption \ref{CompoundAssum} (iv) implies that the one-dimensional compound Poisson coordinate processes also have the intensity $\Lambda$. Furthermore, we have the representation
$$\varphi_{\nu,\alpha}(u) = \exp\left(i\left<u,\alpha\right> + \int_{\mathds{R}^2} (e^{i\left<u,x\right>}-1)\,\nu(dx)\right),\quad u\in\mathds{R}^2$$
with a finite measure $\nu$ and $\alpha\in\mathds{R}^2$. Lemma \ref{InfDecayRateLem} now turns into the following statement:

\begin{lem} \label{CPPDecayLem}
Given a L\'{e}vy triplet $(\Sigma,\nu,\alpha)$ with $\Sigma=0$ and $\nu(\mathds{R}^2)<\infty$. Then it holds
$$\inf_{u\in\mathds{R}^2} |\varphi_{\nu,\alpha}(u)| > 0.$$
\end{lem}
\begin{proof}
We have
$$|\varphi_{\nu,\alpha}(u)| =
\left|\exp\left(\int_{\mathds{R}^2}\left(e^{i\left<u,x\right>}-1\right)\,\nu(dx)\right)\right|\geq
\exp(-2\nu(\mathds{R}^2))>0.$$
\end{proof}
Using this, we can prove the following theorem with the same technique as Theorem \ref{LevyEstTh}. Set $h_n\gl n^{-\frac{1}{2}}$.

\begin{theorem} \label{CompoundJmpEstTh}
It holds under the Assumptions \ref{kernelAssum} and \ref{CompoundAssum} (i)-(iii) the asymptotic
$$\sup_{(a,b)\in\mathfrak{R}} |(a,b)|^2\wedge |(a,b)|^4 |\nu([a,\infty)\times [b,\infty))-\widehat{N}_n(a,b)| = O_{P_{\nu,\alpha}}\left(\frac{(\log n)^5}{\sqrt{n}}\right),\quad n\to\infty.$$
\end{theorem}
\begin{proof}
We only note the changes in the proof of Theorem \ref{LevyEstTh} :

First, we establish
$$P(A_n^c) = P(A_{h_n,n}^c) = P\left(A_{\frac{1}{\sqrt{n}},n}^c\right)\to 0,\quad n\to\infty.$$
Note
$$A_n^c\subseteq\left\{\exists u\in B_{\frac{1}{h_n}}\,:\,\frac{d(\widehat{\varphi}_n,\varphi)}{|\varphi(u)||w(u)|} \geq\frac{1}{2}\right\} \subseteq \left\{n^{-\frac{1}{2}}(\log(e+\sqrt{2}n^{\frac{1}{2}}))^{\frac{1}{2}+\delta}O_P(1)\geq 1\right\},\quad n\in\mathds{N},$$
so that $n^{-\frac{1}{2}} (\log n)^{\frac{1}{2}+\delta}\to 0$ yields $P(A_n^c)\to 0$. Note furthermore
$$\left|\frac{\partial \Psi}{\partial u_k}\right|(u) \lesssim 1,\quad u\in\mathds{R}^2,\quad k=1,2,$$
since
$$\int_{\mathds{R}^2} |x_k|\,\nu(dx) <\infty,\quad k=1,2,\quad \text{cf. Assumptions \ref{CompoundAssum} (i), (ii)}.$$
This is why (\ref{nuEstLastEq}) can be replaced by
\begin{eqnarray}
&&\one_{A_n} \int_{\left[-\frac{1}{h_n},\frac{1}{h_n}\right]^2} (\log(e+|u|))^{2+4\delta} |\mathcal{F} g_{a,b}(-u)|\,\lambda^2(u) O_P(n^{-\frac{1}{2}}) \nonumber\\&=&
\one_{A_n}|(a,b)|^{-2}\vee |(a,b)|^{-4} \nonumber \\&& 
\left(1+\int_{\left[-\frac{1}{h_n},\frac{1}{h_n}\right]^2} (\log(e+|u|))^{2+4\delta}(1+|u_1|)^{-1}(1+|u_2|)^{-1}\,\lambda^2(du)\right)O_P(n^{-\frac{1}{2}}). \label{CPPCopEq1}
\end{eqnarray}
Set w.l.o.g. $\delta\gl\frac{1}{4}$. Then, (\ref{CPPCopEq1}) is not larger than
\begin{eqnarray*}
&&\one_{A_n} |(a,b)|^{-2}\vee |(a,b)|^{-4} (\log n)^3\left(\int_0^{\frac{1}{h_n}} (1+x)^{-1}\,\lambda^1(dx)\right)^2 O_P(n^{-\frac{1}{2}}) \\&=&
\one_{A_n} |(a,b)|^{-2}\vee |(a,b)|^{-4} (\log n)^5 O_P(n^{-\frac{1}{2}}).
\end{eqnarray*}
Again, this proves together with Lemma \ref{convErrorLem} and $|h_n\log h_n| = \frac{\log n}{2\sqrt{n}}$ this theorem.
\end{proof}

Note that we are interested in estimating the copula $C$ of $(\nu(\mathds{R}^2))^{-1}\nu$ instead of the L\'{e}vy copula $\mathfrak{C}$ of $\nu$. Here, we do not need the principle of a L\'{e}vy copula because $\nu$ has no singularity in the origin. Nevertheless, we still treat the origin with our technique as a singularity point. This is due to the fact that we have already developed this technique for the setting of the previous Section \ref{LevCopSec}. However, it is also possible to get some considerable results in case of the compound Poisson with this technique without much extra effort.

\begin{defi}
Let Assumptions \ref{kernelAssum} and \ref{CompoundAssum} hold. Set with the same notation as in Theorem \ref{LevyGeneralTheorem}, but $\delta_n\gl (\log n)^{-1}$
\begin{equation*}
\begin{array}{rclcrclcl}
V_1(x_1) &\gl& \Lambda^{-1} \nu([0,x_1]\times \mathds{R}_+),&& V_2(x_2) &\gl& \Lambda^{-1} \nu(\mathds{R}_+\times [0,x_2]),&& x_1,x_2>0, \\
\widehat{V}_{1,n}^{-1}(u) &\gl& \widehat{U}_{1,n}^{-1}(\Lambda(1-u)),&& \widehat{V}_{2,n}^{-1}(v) &\gl& \widehat{U}_{2,n}^{-1}(\Lambda(1-v)),&& 0<u,v<1
\end{array}
\end{equation*}
and
\begin{eqnarray*}
\widehat{M}_n(a,b) &\gl& 1+\Lambda^{-1}(\widehat{N}_n(a,b)-\widehat{N}_n(a,0)-\widehat{N}_n(0,b)),\quad (a,b)\in\mathfrak{R}, \\
\widehat{C}_n(u,v) &\gl& \widehat{M}_n(\widehat{V}_{1,n}^{-1}(u),\widehat{V}_{2,n}^{-1}(v)),\quad 0<u,v<1.
\end{eqnarray*}
Let furthermore $C$ denote the unique copula of the probability measure $\Lambda^{-1}\nu$, i.e.
$$C(u,v) = M(V_1^{-1}(u),V_2^{-1}(v)),\quad 0< u,v < 1$$
with
$$M(a,b)\gl \Lambda^{-1} \nu([0,a]\times [0,b]) = 1 + \Lambda^{-1}(U(a,b) - U(a,0) - U(0,b)),\quad (a,b)\in \mathfrak{R}.$$
\end{defi}

Note that $V_k\,:\,(0,\infty)\to (0,1)$, $k=1,2$ is a bijection and that its inverse is
$$V_k^{-1}(u) = U_k^{-1}(\Lambda(1-u)),\quad u\in (0,1).$$
\begin{theorem} \label{ComPoiCopTh}
Set $\delta_n\gl (\log n)^{-1}$ and let the Assumptions \ref{kernelAssum} and \ref{CompoundAssum} hold. Then we have for arbitrary and fixed $0<a<b<1$ the asymptotic
$$\sup_{a\leq u,v\leq b} |C(u,v)-\widehat{C}_n(u,v)| = O_{P_{\nu,\alpha}}\left(\frac{(\log n)^{10}}{\sqrt{n}}\right),\quad n\to\infty.$$
\end{theorem}
\begin{proof}
We imitate in the following the proof of Theorem \ref{LevyGeneralTheorem}. First Theorem \ref{CompoundJmpEstTh} yields with
$$\epsilon_n\gl \frac{(\log n)^5}{\sqrt{n}},\quad R(x)\gl x^{-4},\quad x>0$$
and the notation
$$\gamma_n\gl R(\delta_n)\epsilon_n = \frac{(\log n)^9}{\sqrt{n}}\to 0,\quad n\to\infty$$
the asymptotic
$$\sup_{x\geq \delta_n} |\Re_+ \widehat{N}_n(x,0)-U_1(x)| = O_{P_{\nu,\alpha}}(\gamma_n),\quad n\to\infty.$$
Next Corollary \ref{LevCopInvCor} implies, since the Assumption \ref{CompoundAssum} (iv) holds
\begin{equation} \label{CompoundCopEstEq1}
\sup_{a\leq u\leq b} |\widehat{V}^{-1}_{1,n}(u)-V_1^{-1}(u)| = \sup_{\Lambda(1-b) \leq u\leq \Lambda(1-a)} |\widehat{U}^{-1}_{1,n}(u) - U_1^{-1}(u)| = O_{P_{\nu,\alpha}}(\gamma_n),\quad n\to\infty.
\end{equation}
Again as in Theorem \ref{LevyGeneralTheorem}, set
$$u_n\gl V_1\circ \widehat{V}_{1,n}^{-1}(u),\quad v_n\gl V_2\circ \widehat{V}_{2,n}^{-1}(v),\quad a\leq u,v\leq b$$
and note that a copula also is Lipschitz continuous, cf. Nelsen \cite{Nel}[Theorem 2.2.4.]. In particular, we have
\begin{equation} \label{CompoundCopEstEq2}
|C(u,v)-C(u',v')| \leq |u-u'| + |v-v'|,\quad 0<u,u',v,v'<1.
\end{equation}
Together with the mean value theorem and
\begin{equation} \label{CompoundCopEstEqTilde}
\widehat{U}_{1,n}^{-1} \in \mathcal{D}_{\delta_n} \implies \inf_{0<u<1} \widehat{V}^{-1}_{1,n}(u)\geq \inf_{u>0} \widehat{U}_{1,n}^{-1}(u) = \delta_n,
\end{equation}
we have for $a\leq u\leq b$
$$|u-u_n| = |V_1\circ V_1^{-1}(u) - V_1\circ \widehat{V}_{1,n}^{-1}(u)| = |V_1^{-1}(u)-\widehat{V}_{1,n}^{-1}(u)||V_1'(x)|,\quad x\in [V_1^{-1}(a)\wedge \delta_n,\infty).$$
So (\ref{CompoundCopEstEq1}) yields together with the Assumption \ref{CompoundAssum} (iv)
\begin{equation} \label{CompoundCopEstEq3}
\sup_{a\leq u\leq b} |u-u_n| = \delta_n^{-1} O_{P_{\nu,\alpha}}(\gamma_n),\quad n\to\infty.
\end{equation}
Hence (\ref{CompoundCopEstEq2}) and (\ref{CompoundCopEstEq3}) imply
\begin{equation} \label{CompoundCopEstEq4}
\sup_{a\leq u,v\leq b} |C(u,v)-C(u_n,v_n)|\leq \sup_{a\leq u\leq b} |u-u_n| + \sup_{a\leq v\leq b} |v-v_n| = \delta_n^{-1} O_{P_{\nu,\alpha}}(\gamma_n),\quad n\to\infty.
\end{equation}
Furthermore, Theorem \ref{CompoundJmpEstTh} yields together with (\ref{CompoundCopEstEqTilde})
\begin{eqnarray}
\sup_{a\leq u,v\leq b} |C(u_n,v_n)-\widehat{C}_n(u,v)| &=& \sup_{a\leq u,v\leq b} |M(\widehat{V}_{1,n}^{-1}(u),\widehat{V}^{-1}_{2,n}(v)) - \widehat{M}_n(\widehat{V}^{-1}_{1,n}(u),\widehat{V}^{-1}_{2,n}(v))| \nonumber \\&=&
O_{P_{\nu,\alpha}}(\gamma_n),\quad n\to\infty. \label{CompoundCopEstEq5}
\end{eqnarray}
Finally, (\ref{CompoundCopEstEq4}) and (\ref{CompoundCopEstEq5}) prove this theorem.
\end{proof}

\appendix  
\section*{Appendix}
\setcounter{Atheorem}{0}
\setcounter{section}{1}

First, we give a proof of Theorem \ref{Eboundth}. Since empirical process theory is needed, we briefly repeat in the following some definitions and an important result in this context. The respective notations in Van Der Vaart \cite{Vaa} are used:  

Let $(\mathcal{X},\mathcal{A},P)$ be a probability space and let $\mathcal{F}$ be a class of measurable functions $f\,:\,\mathcal{X}\to \mathds{R}$ in $L^2(P)$. Fix any $\epsilon>0$ and let $l,u\,:\,\mathcal{X}\to\mathds{R}$ be two functions in $L^2(P)$ with $\int (l-u)^2\,dP < \epsilon^2$. Then,
$$[l,u]\gl \{f\,:\, \mathcal{X}\to \mathds{R},\quad \text{measurable},\quad l\leq f\leq u\}$$
is called an $\epsilon$-bracket. Denote further with $N_{[]}(\epsilon,\mathcal{F})$ the minimum number of such $\epsilon$-brackets needed to cover $\mathcal{F}$. Note that $l$ and $u$ are not required to belong to $\mathcal{F}$. Next 
$$J_{[]}(\delta,\mathcal{F}) \gl \int_0^\delta \sqrt{\log N_{[]}(\epsilon,\mathcal{F})}\,d\epsilon,\quad \delta >0$$
is called the bracketing integral. Let $(X_i)_{i\in\mathds{N}}$ be a sequence of i.i.d. $\mathcal{X}$-valued and $P$ distributed random variables and set for $n\in\mathds{N}$
$$G_n f\gl \frac{1}{\sqrt{n}} \sum_{i=1}^n(f(X_i)-Ef(X_i)),\quad \|G_n\|_{\mathcal{F}}\gl \sup_{f\in\mathcal{F}} |G_n f|.$$
Then we have

\newtheorem*{VaartCor}{Corollary 19.35 in Van Der Vaart \cite{Vaa}}

\begin{VaartCor} 
For any class $\mathcal{F}$ of measurable functions with envelope function $F$, it holds
\begin{equation} \label{EAstEq}
E^\ast \|G_n\|_\mathcal{F} \lesssim J_{[]}\left(\sqrt{\int |F|^2\,dP}, \mathcal{F}\right).
\end{equation}
\end{VaartCor}
Note that $\lesssim$ means not larger up to a constant which does not depend on $n\in\mathds{N}$. An envelope function $F$ is any $L^2(P)$ function, such that $|f|(x)\leq F(x)$ holds for all $x\in\mathcal{X}$ and $f\in\mathcal{F}$. Finally, observe the star notation $E^\ast$ instead of $E$. This is due to certain measurability problems which are typical in empirical process theory, compare for this the first chapter in Van Der Vaart and Wellner \cite{VaaWel}. Fortunately, we are not concerned with such measurability problems in our case and, thus, can simply write $E$ instead of $E^\ast$ in (\ref{EAstEq}). In general, observe also the helpful monographs of Pollard \cite{Pol} and Dudley \cite{Dud}.

\begin{proof}[Proof of Theorem \ref{Eboundth}]
Write $\Re(z)$ for the real part of a complex number $z$ and $\Im(z)$ for its imaginary part. It suffices to prove the Theorem separately for the real and imaginary part because of
$$\left\|\frac{\partial^l}{\partial u_k^l}A_n(u)\right\|_{L^\infty(w)} \leq \left\|\frac{\partial^l}{\partial u_k^l} \Re (A_n(u))\right\|_{L^\infty(w)} + \left\|\frac{\partial^l}{\partial u_k^l} \Im (A_n(u))\right\|_{L^\infty(w)}.$$
Here, we only treat the real part because the imaginary part can be proven in exactly the same way. We have
$$\Re(A_n(u)) = n^{-\frac{1}{2}} \sum_{t=1}^n (\cos(\left<u,Z_t\right>) - E\cos(\left<u,Z_1\right>)),\quad u\in\mathds{R}^d.$$
Set
$$G_{l,k}\gl \left\{z\mapsto w(u) \frac{\partial^l}{\partial u_k^l} \cos(\left<u,z\right>)\,:\,u\in\mathds{R}^d\right\}$$
with $0\leq l \leq 4, \, 1\leq k\leq d$. Next, it is the crucial idea to apply the above Corollary 19.35 in \cite{Vaa}, i.e. empirical process theory. Here, we choose 
$$f_{l,k}(z) = |z_k|^l,\quad z\in\mathds{R}^d$$
as envelope function for the set $G_{l,k}$. Then the cited corollary implies
\begin{equation} \label{vaaupboundineq}
E \left\|\frac{\partial^l}{\partial u_k^l} \Re(A_n(u))\right\|_{L^\infty(w)} \lesssim J_{[]}\left(\sqrt{E Z_{1,k}^{2l}},\,G_{l,k}\right)
\end{equation}
in the above Notation. With
$$M\gl M(\epsilon,l,k) = \inf\left\{m>0\,:\,E(Z_{1,k}^{2l}\one_{(m,\infty)}(|Z_1|)) \leq \epsilon^2\right\}$$ 
for $\epsilon>0$, define
$$g_j^{\pm}(z)\gl \left(w(u^{(j)})\frac{\partial^l}{\partial u_k^l}\cos\left(\left<u^{(j)},z\right>\right) \pm \epsilon |z_k|^l\right)\one_{[0,M]}(|z|)\pm |z_k|^l\one_{(M,\infty)}(|z|)$$
for later defined fixed points $u^{(j)}\in\mathds{R}^d$. This yields
\begin{eqnarray*}
E(g_j^+(Z_1)-g_j^-(Z_1))^2 &\leq& E\left(4\epsilon^2 Z_{1,k}^{2l}\one_{[0,M]}(|Z_1|) + 4Z_{1,k}^{2l} \one_{(M,\infty)}(|Z_1|)\right) \\ &\leq&
4\epsilon^2(E Z_{1,k}^{2l} +1).
\end{eqnarray*}
Set $C\gl 2\sqrt{E Z_{1,k}^{2l} + 1}$. Then $[g_j^-,g_j^+]$ is a $C\cdot \epsilon$-bracket. Since we are only interested in the finiteness of the right-hand side of (\ref{vaaupboundineq}), we can assume w.l.o.g. $C=1$. Hence, $[g_j^-,g_j^+]$ is an $\epsilon$-bracket. Next we perform some calculations in order to determine the points $u^{(j)}$, such that the upper bound in (\ref{vaaupboundineq}) is finite.
Obviously, 
$$w^1(r)\gl (\log(e+r))^{-\frac{1}{2}-\delta},\quad r\geq 0$$ 
is Lipschitz continuous, so that we have
$$|w(u)-w(v)|=|w^1(|u|)-w^1(|v|)|\leq L||u|-|v|| \leq L|u-v|,\quad u,v\in\mathds{R}^d$$
for some $L>0$. With $u,z\in\mathds{R}^d$ and $|z|\leq M$, we obtain the inequality

\begin{equation} \label{empexpeq1}
\left|w(u)\frac{\partial^l}{\partial u_k^l} \cos(\left<u,z\right>) - w(u^{(j)})\frac{\partial^l}{\partial u_k^l} \cos\left(\left<u^{(j)},z\right>\right)\right| \leq |z_k|^l|u-u^{(j)}|(L+M).
\end{equation}
This yields that (\ref{empexpeq1}) is not larger than
\begin{equation} \label{empexpeq2}
|z_k|^l\min\left\{\sqrt{d}|u-u^{(j)}|_\infty (L+M),w(u)+w(u^{(j)})\right\},\quad z\in \mathds{R}^d\,:\,|z|\leq M.
\end{equation}
Set 
$$U(\epsilon)\gl\inf\left\{u>0\,:\,\sup_{v\in\mathds{R}^d\,:\,|v|\geq u} w(v) \leq \frac{\epsilon}{2}\right\}$$
and
\begin{equation} \label{Jepsdef}
J(\epsilon)\gl\inf\left\{l\in\mathds{N}\,:\,\frac{l\epsilon}{\sqrt{d} (L+M)} \geq U(\epsilon)\right\}.
\end{equation}
Now, we specify the points $u^{(j)}$ as
\begin{equation} \label{ujdef}
u^{(j)} \gl \frac{j\epsilon}{\sqrt{d}(L+M)},\quad j\in\mathds{Z}^d\,:\,|j|_\infty \leq J(\epsilon).
\end{equation}
This choice guarantees that
$$G_{l,k}\subseteq \bigcup_{j\in\mathds{Z}^d\,:\,|j|_\infty \leq J(\epsilon)} [g_j^-,g_j^+].$$
To understand this, fix any $u\in\mathds{R}^d$. If
$$|u|_\infty \leq \frac{J(\epsilon)\epsilon}{\sqrt{d} (L+M)},$$
set
$$j_u\gl\left(\left\lfloor\frac{\sqrt{d}(L+M)}{\epsilon} u_1\right\rfloor,\ldots,\left\lfloor \frac{\sqrt{d}(L+M)}{\epsilon} u_d\right\rfloor\right) \in \mathds{Z}^d$$
with $\lfloor x\rfloor = -\lfloor|x|\rfloor,\,x<0$. Note that $|j_u|_\infty\leq J(\epsilon)$.
Then
$$\sqrt{d}|u-u^{(j_u)}|_{\infty}(L+M)\leq\epsilon$$
and the corresponding function belongs to $[g_{j_u}^-,g_{j_u}^+]$.
If
$$|u|_\infty > \frac{J(\epsilon)\epsilon}{\sqrt{d}(L+M)} \geq U(\epsilon),$$
the corresponding function belongs to 
$$j_u\gl(J(\epsilon),0,\ldots,0)\in\mathds{Z}^d$$ 
because of (\ref{empexpeq2}) and 
$$w(|u|)\leq w(|u|_\infty)\leq\frac{\epsilon}{2},\quad w(|u^{(j_u)}|) = w^1\left(\frac{J(\epsilon)\epsilon}{\sqrt{d}(L+M)}\right)\leq \frac{\epsilon}{2}.$$
Note that (\ref{ujdef}) implies
\begin{equation} \label{Nestineq}
N_{[]}(\epsilon,G_{l,k}) \leq (2J(\epsilon)+1)^d.
\end{equation}
Next, elementary considerations yield
\begin{equation} \label{Mestineq}
M\leq \left(\frac{E|Z_1|^{2l+\gamma}}{\epsilon^2}\right)^\frac{1}{\gamma}.
\end{equation}
Note that (\ref{Jepsdef}) implies
\begin{equation} \label{JEstEq}
J(\epsilon) \leq \frac{U(\epsilon) \sqrt{d} (L+M)}{\epsilon} +1.
\end{equation}
The special shape of $w$, furthermore, yields
$$\log(U(\epsilon)) = \left(\frac{\epsilon}{2}\right)^{-(\delta+\frac{1}{2})^{-1}} + o(1),\quad \epsilon\to 0,$$
so that we have together with (\ref{Nestineq}), (\ref{Mestineq}) and (\ref{JEstEq})
\begin{eqnarray*}
\log(N_{[]}(\epsilon,G_{l,k})) &\leq& d \log(2J(\epsilon)+1) \\ &=&
O\left(\epsilon^{-\left(\delta+\frac{1}{2}\right)^{-1}} + \log\left(\epsilon^{-1-\frac{2}{\gamma}}\right)\right),\quad \epsilon\to 0.
\end{eqnarray*}
As $\left(\delta+\frac{1}{2}\right)^{-1} < 2$, we have established
$$\int_0^{\sqrt{EZ_{1,k}^{2l}}} \sqrt{\log(N_{[]}(\epsilon,G_{l,k}))}\,d\epsilon < \infty$$
and (\ref{vaaupboundineq}) is finite.
\end{proof}

Proposition \ref{FourierProp} is a generalization of the following Proposition \ref{FourierOrgProp} which can be proven with standard results from Fourier analysis.

\begin{Aprop} \label{FourierOrgProp}
Let $f\,:\,\mathds{R}^2\to\mathds{R}$ be a Schwartz space function. Then it holds
\begin{equation} \label{SchwartzIneq}
|\mathcal{F}f|(u) \leq \frac{1}{|u_1u_2|} \int_{\mathds{R}^2} \left|\frac{\partial^2 f}{\partial x_1 \partial x_2}(x)\right|\,\lambda^2(dx),\quad u\in (\mathds{R}^\ast)^2.
\end{equation} 
\end{Aprop}
\begin{proof}
Due to Rudin \cite{Rud}[Theorem 7.4 (c)], it holds the equation
$$u_1 u_2 (\mathcal{F}f)(u) = -\mathcal{F}\left(\frac{\partial^2 f}{\partial x_1 \partial x_2}\right)(u),\quad u\in\mathds{R}^2$$
which implies
$$|\mathcal{F}f|(u) = \frac{1}{|u_1u_2|} \left|\mathcal{F}\left(\frac{\partial^2 f}{\partial x_1\partial x_2}\right)\right|(u) \leq \frac{1}{|u_1u_2|} \int_{\mathds{R}^2} \left|\frac{\partial^2 f}{\partial x_1\partial x_2}(x)\right|\,\lambda^2(dx),\quad u\in (\mathds{R}^\ast)^2.$$
\end{proof}

The situation is more involved in Proposition \ref{FourierProp} since $f$ does not need to be a Schwartz space function. Generally, it is not even a continuous function. The claim of Proposition \ref{FourierProp} states that, in this situation, a similar result as (\ref{SchwartzIneq}) also holds. We only have to take the boundaries into account.

\begin{Aprop} \label{FourierProp}
Let $g\,:\,\mathds{R}^2_+\to\mathds{R}$ be a $\mathcal{C}^2$-function with
$$g\in L^1(\mathds{R}_+^2),\quad \left|\frac{\partial g}{\partial x_j}\right|(x)\lesssim (1+|x|)^{-(1+\epsilon)},\quad\frac{\partial^2 g}{\partial x_1\partial x_2}\in L^1(\mathds{R}_+^2) ,\quad j=1,2$$
for some $\epsilon>0$ and
define a function $f\,:\,\mathds{R}^2\to\mathds{R}$ via
\begin{equation*}
f(x) =
\begin{cases}
g(x),&x\in(\mathds{R}_+^\ast)^2, \\
0,& \text{else.}
\end{cases}
\end{equation*}
Set
$$\Lambda_g\gl |g(0,0)| + \int_{\mathds{R}_+}\left|\frac{\partial g}{\partial x_1}\right|(x_1,0)\,\lambda^1(dx_1) + \int_{\mathds{R}_+} \left|\frac{\partial g}{\partial x_2}\right|(0,x_2)\,\lambda^1(dx_2) +
\int_{\mathds{R}_+^2} \left|\frac{\partial^2 g}{\partial x_1\partial x_2}\right|(x)\,\lambda^2(dx).$$
Then it holds
$$|\mathcal{F}f|(u) \leq \frac{\Lambda_g}{|u_1u_2|},\quad u\in (\mathds{R}^{\ast})^2.$$
\end{Aprop}
\begin{Arem} \label{FourierRem} \rm
Note first that $\mathcal{F}f$ is, of course, independent of the values of $f$ on the negligible set
$$\mathcal{N}_{(0,0)}\gl(\{0\}\times \mathds{R}_+)\cup (\mathds{R}_+\times \{0\}).$$

Furthermore, it holds for every $y\in\mathds{R^2}$
$$\mathcal{F}f(u) = e^{i\left<u,y\right>}\mathcal{F}(f(\cdot +y))(u) ,\quad u\in\mathds{R}^2.$$
Thus, our discontinuity set could also have been
$$\mathcal{N}_{(y_1,y_2)}\gl(\{y_1\}\times [y_2,\infty))\cup ([y_1,\infty)\times \{y_2\})$$
and the above Proposition \ref{FourierProp} remains true. The choice $y=0$ is only due to a simpler notation.
\end{Arem}

\begin{proof}[Proof of Proposition \ref{FourierProp}]
An elementary proof can be done in three steps. In the first step, we approximate $f$ with a sequence of step functions and prove the $L^1$ convergence of the sequence to the given function $f$. The second step calculates the Fourier transforms of those step functions in relation to the Fourier transform of $f$. Finally, the third step combines the results of the first two steps and proves the desired result. We omit the technical details of this proof since they are straightforward, compare for a detailed proof Palmes \cite{Pal}.
\end{proof}

\begin{Acor} \label{gabCor}
Recall 
$$g_{a,b}(x)\gl \frac{1}{x_1^4+x_2^4}\one_{[a,\infty)\times [b,\infty)}(x_1,x_2),\quad (a,b)\in\mathfrak{R},\quad x\in\mathds{R}^2$$
and
$$\mathfrak{R}\gl [0,\infty)^2\backslash\{(0,0)\}.$$
It holds for all $(a,b)\in\mathfrak{R}$ and $u\in\mathds{R}^2$ the inequality
$$|\mathcal{F}g_{a,b}|(u)\lesssim \left(\frac{|(a,b)|^{-4}}{|u_1u_2|}\one_{(\mathds{R}^\ast)^2}(u)\right)\wedge |(a,b)|^{-2}$$
where the constant in the above $\lesssim$ is independent of $(a,b)\in\mathfrak{R}$.
\end{Acor}
\begin{proof}
Note that we have
\begin{eqnarray*}
|\mathcal{F}g_{a,b}|(u) &\leq& \int_{[a,\infty)\times [b,\infty)} \frac{1}{x_1^4+x_2^4}\, \lambda^2(dx) \lesssim \int_{[a,\infty)\times [b,\infty)} \frac{1}{|x|^4}\,\lambda^2(dx) \\ &\lesssim& 
\int_{[|(a,b)|,\infty)} \frac{1}{r^4}r\,\lambda^1(dr) = 2^{-1}|(a,b)|^{-2}
\end{eqnarray*}
where we have used the norm equivalence in $\mathds{R}^2$ and a polar coordinate transformation.

Next, fix any $u\in(\mathds{R}^\ast)^2$ and apply Proposition \ref{FourierProp}. Observe for this purpose that we have
$$\int_{[a,\infty)}\left|\frac{\partial}{\partial x_1} \frac{1}{x_1^4+x_2^4}\right|(x_1,b)\,\lambda^1(dx_1) = -\int_{[a,\infty)}\frac{\partial}{\partial x_1}\frac{1}{x_1^4+b^4}\,\lambda^1(dx_1) = \frac{1}{a^4+b^4}$$
and analogously
$$\int_{[b,\infty)}\left|\frac{\partial}{\partial x_2} \frac{1}{x_1^4+x_2^4}\right|(a,x_2)\,\lambda^1(dx_2) = -\int_{[b,\infty)}\frac{\partial}{\partial x_2} \frac{1}{a^4+x_2^4}\,\lambda^1(dx_2) = \frac{1}{a^4+b^4}.$$
Finally, a similar consideration yields
$$\int_{[a,\infty)\times [b,\infty)} \left|\frac{\partial^2}{\partial x_1\partial x_2} \frac{1}{x_1^4+x_2^4}\right|(x)\,\lambda^2(dx) = \frac{1}{a^4+b^4}\lesssim |(a,b)|^{-4}.$$
\end{proof}

\begin{Alem} \label{trivLevLem}
Finally, we state two simple but useful facts.
\begin{itemize}
\item[(i)]
We have with $u\in\mathds{R}^2$ and $|u_1|,|u_2|\geq 1$ the inequality
\begin{equation} \label{u1u2Eq}
\frac{1}{|u_1u_2|}\leq \frac{4}{(1+|u_1|)(1+|u_2|)}.
\end{equation}
\item[(ii)]
It holds for every $0<h<\frac{1}{2}$
$$\int_{\mathds{R}} \frac{\min(h|z|,1)}{(1+|z|)^2}\,dz \lesssim |h\log h|.$$ 
\end{itemize}
\end{Alem}
\begin{proof}
The proof is elementary and therefore omitted.
\end{proof}

{\bf Acknowledgements.} The financial support of the Deutsche Forschunsgemeinschaft (FOR 916, project B4) is gratefully acknowledged.

\end{document}